\documentclass[12pt,a4paper]{article} 
\usepackage{amsmath,amsfonts,amsthm,pstricks,graphicx,algorithm,algorithmic}
\usepackage{verbatim}
\numberwithin{equation}{section}
\numberwithin{table}{section}
\numberwithin{figure}{section}
\title{Iterative methods\\
for shifted positive definite  linear systems\\
and time discretization of the heat equation}

\author{William McLean\thanks{School of Mathematics and Statistics,
The University of New South Wales,
Sydney 2052, Australia
(\texttt{w.mclean@unsw.edu.au})}
\and
Vidar Thom\'ee\thanks{Mathematical Sciences,
Chalmers University of Technology and University of Gothenburg,
S-412~96 Gothenburg, Sweden
(\texttt{thomee@chalmers.se})}}
\newcommand{\rhosz}{\breve\rho_z}
\newcommand{\varphisz}{\breve\varphi_z}
\newcommand{\epsilonsz}{\breve\ep_z}
\newcommand{\nusz}{\breve\nu_z}
\newcommand{\alphasz}{\breve\alpha_z}
\newcommand{\Lambdasz}{\breve\Lambda_z}

\newcommand{\epsilonso}{\breve\ep_0}

\newcommand{\alphaso}{\breve\alpha_0}

\newcommand{\tb}[1]{|\hskip -.08em|\hskip-.08em|#1|\hskip-.08em|\hskip-.08em|}
\newcommand{\db}[1]{|[#1]|}

\newcommand{\al}{\alpha}
\newcommand{\be}{\beta}
\newcommand{\ga}{\gamma}
\newcommand{\Ga}{\Gamma}

\newcommand{\si}{\sigma}
\newcommand{\Si}{\Sigma}
\newcommand{\ep}{\varepsilon}
\newcommand{\la}{\lambda}
\newcommand{\La}{\Lambda}

\newcommand{\om}{\omega}
\newcommand{\Om}{\Omega}
\newcommand{\ka}{\kappa}
\newcommand{\for}{\quad\text{for}\ }

\newcommand{\andy}{\quad\text{and}\ }
\newcommand{\with}{\quad\text{with}\ }

\newcommand{\where}{\quad\text{where}\ }

\newcommand{\wt}{\widetilde}
\newcommand{\wh}{\widehat}

\newcommand{\fy}{\varphi}
\newcommand{\Lapl}{\mathcal{L}}
\newcommand{\A}{\mathcal{A}}
\newcommand{\B}{\mathcal{B}}
\newcommand{\F}{\mathcal{F}}
\renewcommand{\H}{\mathcal{H}}

\renewcommand{\H}{\mathcal{H}}
\renewcommand{\L}{\mathcal{L}}
\newcommand{\M}{\mathcal{M}}
\newcommand{\I}{\mathcal{I}}

\newcommand{\D}{\mathcal{D}}
\newcommand{\E}{\mathcal{E}}
\newcommand{\T}{\mathcal{T}}

\newcommand{\R}{\mathbb{R}}
\newcommand{\C}{\mathbb{C}}
\newcommand{\Poly}{\mathbb{P}}

\newcommand{\vecspan}{\operatorname{span}}

\renewcommand{\Re}{\operatorname{Re}}
\renewcommand{\Im}{\operatorname{Im}}
\renewcommand{\S}{\mathcal{S}}

\renewcommand{\hat}{\widehat}

\newcommand{\vecg}{\boldsymbol{g}}
\newcommand{\vecp}{\boldsymbol{p}}
\newcommand{\vecr}{\boldsymbol{r}}
\newcommand{\vecu}{\boldsymbol{u}}
\newcommand{\vecv}{\boldsymbol{v}}
\newcommand{\vecw}{\boldsymbol{w}}
\newcommand{\iprod}[1]{\langle#1\rangle}

\renewcommand{\a}{\mathsf{a}}
\renewcommand{\b}{\mathsf{b}}
\renewcommand{\c}{\mathsf{c}}
\renewcommand{\d}{\mathsf{d}}
\newcommand{\abar}{\overline{\mathsf{a}}}
\newcommand{\bbar}{\overline{\mathsf{b}}}
\newcommand{\cbar}{\overline{\mathsf{c}}}
\newcommand{\dbar}{\overline{\mathsf{d}}}
\newcommand{\ff}{\mathsf{f}}
\newtheorem{theorem}{Theorem}[section]
\newtheorem{lemma}[theorem]{Lemma}
\newtheorem{proposition}[theorem]{Proposition}
\newtheorem{corollary}[theorem]{Corollary}
\begin{document}
%%%%%%%%%%%%%%%%%%%%%%%%%%%%%%%%%%%%%%%%%%%%%%%%%%%%%%%%%%%%%%%%%%%%%%%%
\maketitle
%%%%%%%%%%%%%%%%%%%%%%%%%%%%%%%%%%%%%%%%%%%%%%%%%%%%%%%%%%%%%%%%%%%%%%%%
\begin{abstract}

In earlier work we have studied a method for discretization in time
of a parabolic problem which consists in representing the exact solution
as an integral in the complex plane and then applying a quadrature 
formula to this  integral. In application to a spatially semidiscrete
finite element version of the parabolic problem, at each quadrature point one
then needs to solve a linear algebraic system
having a positive definite matrix
with a complex shift, and in this paper we study iterative methods
for such systems. We first consider the basic and a preconditioned version
of the Richardson algorithm, and then a conjugate gradient method as 
well as a preconditioned version thereof.  
\end{abstract}

%%%%%%%%%%%%%%%%%%%%%%%%%%%%%%%%%%%%%%%%%%%%%%%%%%%%%%%%%%%%%%%%%%%%%%%%
\section{Introduction}
Let $V$ be a complex finite-dimensional inner product space,
and let $A$ be a positive definite Hermitian linear 
operator in $V$, with spectrum~$\sigma(A)$. We shall consider 
iterative methods for the linear equation
\begin{equation}\label{1.0}
zw+Aw=g,\where z=x+iy\not\in -\si(A).
\end{equation}
Such equations, with a complex shift $z$ of the positive definite
operator $A$, need to be solved in a method for discretization
in time of parabolic equations, based on Laplace transformation
and quadrature, which has been studied recently, as will be
made more specific below.  Equations of the form~\eqref{1.0} arise
also from the spatial discretization of the Helmholtz 
equation, cf.~\cite{Freund1992}, but in that context the $z$-values typically
of interest differ from those we wish to consider; for our application to
the heat equation, $\arg z$ is bounded away from~$\pm\pi$.
In this paper we  shall consider a basic Richardson iteration and 
a conjugate gradient (CG) method for \eqref{1.0}, 
as well as preconditioned versions of these methods.  Another approach,
not discussed here, is to reformulate the complex
linear system as an equivalent real one with twice as many equations
and unknowns, cf., e.g.,
%  See, for instance, the recent paper of Benzi~and 
\cite{BenziBertaccini2008} and the list of references therein.
\medskip

We begin by sketching the time discretization method referred to above.
In a  complex Hilbert space $\H$
we consider the initial-value problem 
\begin{equation}\label{ivp}
u_t+Au=f(t),\for t>0,\with u(0)=u_0,
\end{equation}
where $A$ is a positive definite Hermitian operator in $\H$.
To represent its solution we apply the Laplace transform, writing
\[
w(z)=\hat u(z)=\Lapl u(z):=\int_0^\infty e^{-zt}u(t)\,dt,\quad \Re z>0.
\]
Under appropriate assumptions on $f(t)$ we then have formally
\[
(zI+A)w(z)=w_0+\hat f(z)=:g(z),
\]
or, multiplying by  the resolvent of $-A$,
\[
w(z)=R(z)g(z),\where
R(z):=(zI+A)^{-1}.
\]
Applying the inverse Laplace transform, we obtain
\[
u(t)=(\Lapl^{-1}w)(t)=\frac1{2\pi i}\int_{\Ga_\om} e^{zt} w(z) \,dz,
\]
for $\Ga_\om:=\{z:\ \Re z=\om\}$ and $\om>0$.
With $\fy\in(\tfrac12\pi,\pi)$ and $\Ga$
a new contour in $\Si_\fy=\{z:\ |\arg z|<\fy\}$, 
homotopic with $\Ga_\om$, we may write
\[
u(t)=\frac1{2\pi i}\int_{\Ga} e^{zt}
w(z)\,dz.
%\where w(z)= R(z)g(z),\ R(z):=(zI+A)^{-1}.
\]
A suitable parametrization of $\Ga$, written $z=z(\xi)$ for~$\xi\in \R$, 
yields
\begin{equation}\label{1.inte}
u(t)=\int_{-\infty}^\infty v(\xi,t)\,d\xi,
\where
v(\xi,t):=\frac{1}{2\pi i}\,e^{z(\xi)t}
	w\bigl(z(\xi)\bigr)z'(\xi).
\end{equation}
We assume  $\Re z(\xi)\to -\infty$ as $|\xi|\to\infty$
so that $e^{z(\xi)\,t}\to 0$, for $t>0$.

We now define an approximate solution 
of \eqref{ivp} by  means of an equal-weight  quadrature rule,
applied to the integral in \eqref{1.inte},
\begin{equation}\label{1.app}
U_q(t):=k\sum_{j=-q}^q v(\xi_j,t)
=\frac k{2\pi i}\sum_{j=-q}^q e^{z_jt}w(z_j)z'_j,
\end{equation}
where, for an appropriate~$k>0$, we have set
\begin{equation}\label{1.appi}
\xi_j:=jk\in \R,\quad z_j:=z(\xi_j),\quad z_j':=z'(\xi_j),
\for |j|\le q.
\end{equation}
To compute $U_q(t)$, we need to  solve the $2q+1$ ``elliptic" equations
\begin{equation*}\label{zeq}
(z_jI+A)w(z_j)= g(z_j),\for |j|\le q.
\end{equation*}
These equations are independent, and may thus be solved
in parallel.
We note that the $w(z_j)$ determine $U_q(t)$ for all $t>0$, but 
we can expect an accurate approximation only for~$t$ in some restricted
interval that depends on the choice of the quadrature step~$k$ and of
the parametric representation~$z(\xi)$.

In our presentation  we shall
follow the analysis of \cite{biv06}. Specifically, we use 
for $\Ga$ the left branch of the hyperbola
$(x-1)^2-y^2=1$ in the complex plane,
parametrized by
\begin{equation*}\label{eq: standard contour}
z(\xi)=1-\cosh \xi+i\,\sinh \xi,\quad  \xi\in \R,
\end{equation*}
and take $k=\log q/q$ for the step size in \eqref{1.app}.
This means that 
\[
z_j=x_j+i\,y_j=1-\cosh\biggl(\frac{j\log q}q\biggr)
	+i\,\sinh\biggl(\frac{j\log q}{q}\biggr),
\for |j|\le q.
\]
In particular, $z_q=1-(q+q^{-1})/2+i\,(q-q^{-1})/2\approx -q/2+i\,q/2$ for
large~$q$.
Under the appropriate assumptions about the data of the problem
we then have the error estimate, see \cite{biv06}, 
with $0<t_0<T<\infty$,
\[
\|U_{q}(t)-u(t)\|\le 
C_{t_0,T}(u_0,f)\,e^{-c\, q/\log q},\for t\in[t_0, T].
\]

We now want to apply this time discretization scheme to the semidiscrete
finite element approximation of the heat equation, with elliptic
operator~$Lu=-\nabla\cdot(a\nabla u)$, and
consider thus the initial boundary-value problem for $u=u(x,t)$,
\begin{equation}
\label{1.heat}
\begin{aligned}
u_t+Lu&=f(\cdot,t),&&\text{in $\Omega$,}\quad 
\text{with $u=0$ on $\partial\Om$,}\quad\text{for $t>0$,}\\
u(0,\cdot)&=u_0,&&\text{in $\Omega$,}
\end{aligned}
\end{equation}
where, for simplicity, we will assume that the diffusivity~$a$ is a 
(positive) constant, and that $\Omega$ is a convex polygonal domain in~$\R^2$.
This problem is the special case of \eqref{ivp} with
$\H=L_2(\Om)$ and $A=L$, taking $D(L)=H^2(\Om)\cap H_0^1(\Om)$.

Let $\{V_h\}\subset H_0^1(\Om)$
be a family of piecewise linear finite element spaces,
based on a family of regular  triangulations $\T_h=\{\tau\}$ of $\Om$.
With $(v,w)=\int_\Om v\,\bar w\,dx$,
the standard Galerkin,
spatially semidiscrete approximation of \eqref{1.heat}  is
\begin{align*}
(u_{h,t},\chi)+a(\nabla u_{h},\nabla \chi)
	=(f,\chi),\quad\forall\chi\in V_h,\ t>0,
\with u_h(0)=u_{0h},
\end{align*}
where, with
$P_h: L_2(\Om)\to V_h$ the $L_2$-projection onto $V_h$, we may take,
e.g., $u_{0h}=P_hu_0$.
Introducing the discrete elliptic operator~$L_h:V_h\to V_h$, defined by
\[
(L_h\psi,\chi)=a(\nabla\psi,\nabla\chi),\quad\forall\psi,\chi\in V_h,
\]
the spatially semidiscrete initial-value problem may also be written
\[
u_{h,t}+L_h u_h=P_hf(\cdot,t),\for t>0,\with
u_h(0)=P_hu_0,
\]
which is of the form 
\eqref{ivp} with $\H=V_h,$ equipped with the $L_2$ inner product,
and $A=L_h$.
The fully discrete solution defined by our above time discretization
method \eqref{1.app} now takes the form
\begin{equation}\label{1.dissol}
U_{q,h}(t):=\frac{k}{2\pi i}\sum_{j=-q}^qe^{z_jt} w_h(z_j) \,z'_j,
\end{equation}
with $z_j, z_j'$ as in \eqref{1.appi} and where
the $w_h(z_j)$ are derived from
\begin{equation*}\label{zeqh}
(z_jI+L_h)w_h(z_j)= P_hg(z_j),\for |j|\le q,
\end{equation*}
or, in weak form,
\begin{equation}\label{femsys}
z_j\bigl(w_h(z_j),\chi\bigr)
        +a\bigl(\nabla w_h(z_j),\nabla\chi\bigr)
        =\bigl(g(z_j),\chi\bigr),
        \quad \forall\ \chi\in V_h.
\end{equation}
As before, these problems may be solved in parallel.
We note that they are special cases
of~\eqref{1.0}, with $V=V_h$ and $A=L_h$.
Under appropriate assumptions on the data~\cite{biv06} the error in the
fully discrete solution may be bounded as
\begin{equation}\label{1.err}
\|U_{q,h}(t)-u(t)\|\le 
C_{t_0,T}(u_0,f)(h^2+e^{-c\, q/\log q}),
\for t\in [t_0,T]\subset(0,\infty). 
\end{equation}

To express~\eqref{femsys} in matrix form, let
$\{P_i\}_{i=1}^N$ be the interior nodes of $\T_h$ and $\{\Phi_i\}_{i=1}^N$ the
associated nodal basis functions, so that 
$v\in V_h$ may be written as $v=\sum_{i=1}^N\vecv_i\Phi_i$ with 
$\vecv_i:=v(P_i)$.
Let $\M=(m_{il})$ and $\S=(s_{il})$ be the mass and stiffness matrices, where
$m_{il}:=(\Phi_i,\Phi_l)$ and $s_{il}:=a(\nabla\Phi_i,\nabla\Phi_l)$,
respectively. With 
$w=w_h(z_j)=\sum_{i=1}^N\vecw_i\Phi_i,$
equation~\eqref{femsys} is then equivalent to
\begin{equation}\label{1.matsg}
z_j\M\vecw+\S\vecw=\vecg 
\quad\text{or}\quad
z_j\,\vecw+\M^{-1}\S\vecw=\M^{-1}\vecg,
\end{equation}
where the components of the load vector are 
$\vecg_i=\bigl(g(z_j),\Phi_i\bigr)$.
The second equation in~\eqref{1.matsg} is of the form~\eqref{1.0}
with $A\vecv=\M^{-1}\S\vecv$ and $g=\M^{-1}\vecg$.
However, instead of the standard
unitary inner product~$\iprod{\vecv,\vecw}=\sum_{i=1}^N v_i\bar w_i$,
we equip $V=\C^N$ with~$(\vecv,\vecw)=\iprod{\M\vecv,\vecw}$ so that $A$
is Hermitian: $(A\vecv,\vecw)=\iprod{\M A\vecv,\vecw}=\iprod{\S\vecv,\vecw}$.
In our study of iterative methods for~\eqref{1.matsg}, we develop the theory
for an abstract operator~$A$ satisfying our assumptions, and discuss 
separately the practical implications for the specific choices
$A=L_h$~and, especially, $A=\M^{-1}\S$.

As an alternative to the standard Galerkin method we may consider
the lumped mass modification, in which the mass matrix $\M$
is replaced by a diagonal matrix $\D$; we refer to \cite{th06} for details.

For any Hermitian operators $A$~and $B$ in~$V$, with $B$ positive definite,
we will write $\lambda_j=\lambda_j(A,B)$ for
the $j$th generalized eigenvalue of~$A$ with respect to~$B$, that is,
$Av_j=\lambda_jBv_j$ with $v_j\ne0$.  We order these
eigenvalues so that $\lambda_1\le\lambda_2\le\cdots\le\lambda_N$, and
use the abbreviation $\lambda_j(A)=\lambda_j(A,I)$

The program for time discretization of parabolic
equations sketched above was initiated in Sheen,
Sloan and Thom\'ee \cite{SST99,SST03},
and continued in Gavrilyuk and Makarov \cite{GM}, McLean, Sloan, and 
Thom\'ee \cite{biv06} and McLean and Thom\'ee \cite{bivi04,bivi07,bivi08}, 
cf.~also Thom\'ee \cite{th05}, and the error 
in~\eqref{1.dissol} was analyzed in both $L_2(\Om)$~and $L_\infty(\Omega)$,
under various assumptions on the data of the problem. 
In the latter papers also fractional order diffusion equations were treated.
  
In these papers the analysis was illustrated by numerical examples. These 
were carried out in simple cases, in one space dimension and also in the
case of a square spatial domain in two dimensions, and direct solvers 
were used for the linear system~\eqref{femsys}.  However,  even though 
powerful direct solvers are available, for large size problems in more
complicated geometries, particularly in 3D, it may be natural to apply 
iterative methods, and our purpose in this paper is therefore to 
begin a study of such methods for equations of the form~\eqref{1.0}, 
with application to the heat equation in mind.
Some preliminary results on this problem were sketched in~\cite{SST03},
using the Richardson iteration algorithm
for \eqref{femsys} and for a preconditioned form
of this equation, and in Section~\ref{sec: Richardson} below we extend and
improve these results.  

From a knowledge of the extremal eigenvalues of~$A$,
we can determine the optimal value
of the complex acceleration parameter, optimal in the sense of
minimizing the error reduction factor of the Richardson iteration.
For the finite element problem on quasiuniform
triangulations with maximal mesh-size~$h$,
the basic Richardson method converges slowly, with the error in the $n$th
iterate bounded by $(1-ch^2)^n$, for $c>0$ depending on~$z$, but with
the convergence rate improving with growing $|z|$.  
We also study preconditioned versions of this method, first using 
the special preconditioner~$B_z=(\mu_zI+A)^{-1}$, where
$\mu_z>-\lambda_1(A)$, which may be analyzed in the same
way as the basic method, and we show that the error reduction factor is
bounded away from~$1$ as~$\lambda_N\to\infty$.
We then consider a general preconditioner~$B_z$, and prove geometric
convergence in the norm~$\db{v}:=(B_z^{-1}v,v)^{1/2}$, where the 
acceleration parameter is defined in terms of bounds for the spectrum
of~$B_z(\mu_zI+A)$.

In Section~\ref{sec3} we  analyze a CG method, which does not involve
choosing an acceleration parameter. Generalizing the usual 
convergence analysis to allow the complex shift of~$A$ in~\eqref{1.0},
we show geometric convergence of the iterates $w_n$,
which follows from the error bound
\begin{equation}\label{1.errbd}
\tb{w_n-w}\le \frac{\sec(\frac12\arg z)}{
|T_n(s_z)|}\,
\tb{w_0-w},
\with  s_z:=\frac{\la_1+\la_N+2z}{\la_N-\la_1},
\end{equation}
where $\tb{v}^2:=|z|\|v\|^2+(Av,v)$ and
$T_n$ is the Tchebyshev polynomial of degree~$n$,
and where $\lambda_j=\lambda_j(A)$.  Since 
$T_n(s_z)=\tfrac12(\eta_z^{n}+\eta_z^{-n})$ with $|\eta_z|<1$, this
indicates geometric convergence with rate~$|\eta_z|^n$.  For the finite
element problem discussed above, we find that $|\eta_z|\le 1-ch$ 
with~$c>0$ depending on~$z$, giving a better convergence rate than
Richardson iteration.

If the equation is preconditioned with 
$B_z=(\mu_zI+A)^{-1}$, for appropriate~$\mu_z$, and if we let
$\wt z=(z-\mu_z)^{-1}$, then 
the preconditioned equation is equivalent to $\wt z w+B_zw=\wt zB_z g$,
which again has the form~\eqref{1.0}, and
a similar convergence result holds with an error reduction factor bounded
away from~$1$ as~$\lambda_N\to\infty$.

It is natural to consider more general
preconditioners also for the CG iteration. The preconditioned equation
$zB_zw+B_zAw=B_zg$ is again equivalent to an equation of the form~\eqref{1.0},
namely, $zv+B_z^{1/2}AB_z^{-1/2}v=B_z^{1/2}g$, where $v=B_z^{1/2}w$
and the transformed operator~$B_z^{1/2}AB_z^{-1/2}$ is Hermitian and 
positive-definite with respect to the inner produce $[v,w]=(B_z^{-1}v,w)$.  
However, computing the action of~$B_z^{\pm1/2}$ will usually be costly, 
so we instead work with the preconditioned equation in its original form.  
Although the error is still optimal in a certain
sense, we are not able to show
a precise error bound of the type~\eqref{1.errbd}.

Section~\ref{sec: implementation} develops the algorithmic 
implementation of the CG method. For the basic method, 
the successive iterates satisfy a three term recursion relation.
The same is true of the preconditioned method for the special
choice~$B_z=(\mu_zI+A)^{-1}$, but not necessarily for a more general
preconditioner.

Use of an iterative solver means that we compute an 
approximation~$\wt w_h(z_j)$ in place of the true finite element 
solution~$w_h(z_j)$, so that in place of~\eqref{1.dissol} we obtain
\[
\wt U_{q,h}(t):=\frac{k}{2\pi i}\sum_{j=-q}^q e^{z_jt}\wt w_h(z_j)z_j'.
\]
If $\|\wt w_h(z_j)-w_h(z_j)\|\le \ep_j$, then 
\begin{equation}\label{eq: E(t)}
\E(t):=\|\wt U_{q,h}(t)-U_{q,h}(t)\|
\le\frac{k}{2\pi}\sum_{j=-q}^q \ep_j e^{x_j t}\,|z_j'|,
\end{equation}
and we may use this estimate as the basis for a stopping criterion.
In view of the error estimate~\eqref{1.err} we see that it is
desirable to choose the solver tolerance~$\ep_j$ in such a way that 
$\E(t)\le C(h^2+e^{-cq/\log q})$.  The presence of the 
factor~$e^{x_jt}|z_j'|$ allows $\ep_j$ to increase with~$|j|$; see
\eqref{eq: epsilon_j delta} below and remember that $x_j<0$.

In the final Section \ref{sec5} we illustrate our error analysis
by numerical calculations in a concrete case of \eqref{1.heat},
and discuss how to choose the parameters to balance the contributions
to the error of the discretizations in space and time and in
the iterative procedure.
%%%%%%%%%%%%%%%%%%%%%%%%%%%%%%%%%%%%%%%%%%%%%%%%%%%%%%%%%%%%%%%%%%
\section{Iteration  algorithms of Richardson type}
\label{sec: Richardson}
 
We now assume, as in \eqref{1.0},
that $A$ is a  positive definite Hermitian operator 
in a finite-dimensional complex inner product space~$V$,
with extremal eigenvalues $\lambda_1=\lambda_1(A)$ 
and $\lambda_N=\lambda_N(A)$, and for brevity put $A_z:=zI+A$.
In this section, following \cite{SST03},
we consider first the basic Richardson iteration 
with acceleration parameter~$\al\in \C$, applied to~$A_zw=g$,
\begin{equation}
\label{6.5}
%w^{n+1}=(I-\alpha(z\,I+A))w^n+\alpha g,\for n\ge0,
w^{n+1}=(I-\alpha A_z)w^n+\alpha g,\for n\ge0,
\with w^0\ \text{given}.
\end{equation}
The error reduction in each time step is then described by
the inequality
$$
%\|w^{n+1}-w\|\le \|I-\alpha(zI+ A)\|\,\|w^{n}-w\|,
\|w^{n+1}-w\|\le \|I-\alpha A_z\|\,\|w^{n}-w\|,
$$
and since $A_z$ is a normal operator in~$V$,
\begin{equation}\label{2.norm0}
\|I-\alpha A_z\|=\max_{\la\in \si(A)}|1-\alpha(z+\la)|.
\end{equation}

In~\eqref{6.5}, in addition to choosing $w^0$, the issue is
to select $\alpha\in\C$ so that
the norm in~\eqref{2.norm0} is as small as possible.
For $z=0$, as is well known,
the optimal choice of~$\alpha$ is $2/(\la_1+\la_N)$,
which gives
\begin{equation*}\label{2.z0}
\|I-\alpha A\|= \frac{\ka(A)-1}{\ka(A)+1},
\where \ka(A):=\frac{\la_N}{\la_1}.
\end{equation*}
When $A=L_h$ is based on a quasi-uniform family of triangulations $\T_h$ 
we have $\ka(A)=O(\la_N )=O(h^{-2})$ and hence,
 in this case,
\begin{equation}\label{sa}
\|I-\alpha A\|\le 1-ch^2,\with c>0.
\end{equation}

To determine an optimal $\al$ in \eqref{2.norm0}, we shall have use for
the following lemma.

\begin{lemma}\label{l2.1}
Let $\a$, $\b\in\C$ be nonproportional, and 
$[\a,\b]\subset\C$ the line segment with endpoints $\a$~and $\b$. Set
\begin{equation*}\label{2.Fdef}
F(\al):= \max_{\la\in[\a,\b]} | 1- \alpha  \la|,\where \al\in\C. 
\end{equation*}
Then $F(\al)<1$ for suitable $\al$, and $F(\al)$ is minimized by
\begin{equation*}\label{2.6}
\al=\frac{1}{\c+s\d},\quad\text{where $\c:=\tfrac12(\a+\b)$, $\d:=i(\b-\a)$,}
\end{equation*}
and where $s\in \R$ minimizes the real rational function
\begin{equation*}\label{2.7}
R(s):=\Bigl|1-\frac{\a}{\c+s\d}\Bigr|^2
=\frac{|\d|^2s^2+2s\Re\bigl((\c-\a)\dbar)+|\c-\a|^2}
{|\d|^2s^2+2s\Re(\c\dbar)+|\c|^2}.
\end{equation*}
The minimizing value of~$s$ is given,
with the $\pm$ sign being that  of~$\Re(\a\dbar)$, by
\[
s_{\min}=-\ff_1\pm\sqrt{\ff_1^2-\ff_2},
\]
where
\[
\ff_1:=\frac{2\Re(\a\cbar)-|\a|^2}{2\Re(\a\dbar)}
\quad\text{and}\quad
\ff_2:=\frac{2\ff_1\Re(\c\dbar)-|\c|^2}{|\d|^2}.
\]
\end{lemma}
\begin{proof}
We first note that $\al$ may be chosen so that $F(\al)<1$.  In fact,
we may  first rotate the line segment $[\a,\b]$ around the origin so 
that it becomes parallel to and to the right of
the imaginary axis, which determines $\arg\al$,
and then  shrink the line segment thus rotated so
that it comes inside the disk~$|z-1|<1$, giving  $|\al|$.

For $\al$ to be optimal, we must have
$|1-\al\a|=|1-\al\b|$, and thus also
$|1/\al-\a|=|1/\al-\b|$. Therefore, $1/\al$
has to be chosen on the line in~$\C$ through
the midpoint $\c=\tfrac12(\a+\b)$, which is perpendicular to~$\b-\a$, 
or has the direction of~$\d=i(\b-\a)$, so that
$1/\al=\c+s\d$, or $\al=1/(\c+s\d)$, with~$s\in \R$, and 
$R(s)=\bigl|F\bigl(\alpha(s)\bigr)\bigr|^2$.  

Since $\d\ne0$ we have $R(s)\to1$ as~$s\to\pm\infty$, and if
$\Re(\a\dbar)>0$ (${}<0$, respectively) then $R(s)<1$ (${}>1$, respectively)
for large~$s>0$.  Note that since $\a=\a_1+i\a_2$ and $\b=\b_1+i\b_2$ are 
nonproportional, we have
$\Re(\a\dbar)=-\Re\bigl(i\a(\bbar-\abar)\bigr)=
-\Re(i\a\bbar)=\a_2\b_1-\a_1\b_2\ne0$.
A simple calculation shows
\[
R'(s)=\frac{2\Re(\a\dbar)|\d|^2(s^2+2\ff_1s+\ff_2)}
{(s^2|\d|^2+ 2s\Re(\cbar\d)+ |\c|^2)^2},
\]
so that $R(s)$ has just one  maximum and one minimum,
with the maximum to the left of the minimum if and only
if $\Re(\a\dbar)>0$.
\end{proof}

We are now ready to show the following.

\begin{theorem}\label{th1}
Let $z=x+iy$ with $\arg z\in (-\pi,\pi)$ and determine $\alpha=\al_z$ by
taking $\a=z+\la_1$ and $\b=z+\la_N$ in Lemma~\ref{l2.1}.  Then, 
for~$\lambda_N$ sufficiently large, the error reduction factor in~\eqref{6.5}
satisfies
\[
\ep_z:=\min_\al \| I - \alpha A_z\|=|1-\al_z\a|\le
1-c\la_N^{-1},\with c=c(z,\lambda_1)>0.
\]
\end{theorem}
\begin{proof}
With the notation of Lemma \ref{l2.1} we have, for~$\la_N\to\infty$, 
\[
2\Re(\a\cbar)=(x+\la_1)\la_N+O(1),\quad 
\Re(\a\dbar) =y\la_N+O(1),\quad
\Re(\c\dbar) =y\la_N+O(1).
\]
Hence, putting $s_\pm:=-\big(-x-\la_1\pm\sqrt{(x+\la_1)^2+y^2}\big)/(2y)$,
\[
\ff_1=\frac{x+\la_1}{2y}+O(\lambda_N^{-1}),\quad 
\ff_2=-\frac14+O(\lambda_N^{-1}),\qquad
s_{\min}=s_\pm+O(\lambda_N^{-1}),
\]
and it follows by Lemma~\ref{l2.1} that 
$\al_z=(\frac12+is_\pm)^{-1}\lambda_N^{-1}+O(\lambda_N^{-2})$.  Therefore,
\[
|1-\al_z\a|^2=1-2\Re(\al_z\a)+|\alpha_z\a|^2
	=1-\beta\la_N^{-1}+O(\lambda_N^{-2}),
\]
where
\begin{align*}
\beta&=\beta(z,\lambda_1)
	=2\Re\biggl(\frac{x+\lambda_1+iy}{\tfrac12+is_\pm}\biggr)\\
	&=\frac{x+\lambda_1+2ys_\pm}{\tfrac14+s_\pm^2}
	=\pm\frac{\sqrt{(x+\lambda_1)^2+y^2}}{2y(\tfrac14+s_\pm^2)}>0,
\end{align*}
since the sign in $\pm$ is that of $y$, and the desired estimate follows
for~$\lambda_N$ sufficiently large.
\end{proof}

When, as above $A=L_h$, with~$\{\T_h\}$
quasiuniform, so that $\la_N\approx ch^{-2}$, 
the error bound is  of the same form as in \eqref{sa},
except that now the constant $c$ depends on $z.$ 

The rate of convergence shown in Theorem \ref{th1}
is too slow for the iteration to be of practical use.
In Table~\ref{tab1}, we show the values of the 
parameter~$\alpha=\rho e^{-i\varphi}$ and
the error reduction factor~$\ep_z$ given by Theorem~\ref{th1},
with~$z=z_j$ on the hyperbola $(x-1)^2-y^2=1$, for even
$j$ in the range $0\le j\le q=20$.  Here,
the operator~$A$ is from the model problem described in
Section~\ref{sec5}, for which $\lambda_1\approx1$
and $\lambda_N\approx4,000$.  
\medskip

One way to improve the convergence  of the iterative method \eqref{6.5},
considered briefly in \cite{SST03},
is to precondition the linear system by multiplication by a positive definite
Hermitian operator~$B_z$, which, in contrast to the choice
in~\cite{SST03},  we here allow to depend on $z$.  Rewriting
\eqref{1.0} as
\begin{equation}\label{2.10}
G_zw= \wt g_z:=B_zg, \where \ 
%G_z:=B_z(zI+A),
G_z:=B_zA_z,
\end{equation}
the Richardson iteration algorithm becomes
\begin{equation}\label{2.2}
w^{n+1}=(I-\al G_z)w^n+\al \wt g_z.
\end{equation}

We first  consider the special preconditioner~$B_z=(\mu_zI+A)^{-1},$ 
where $\mu_z>-\lambda_1$.
One could choose, for example, $\mu_z=0$, as in \cite{SST03}, or
$\mu_z=|z|$. For $\mu_z=0$ we have $B_z=A^{-1}$, independently of $z$,
and for $\mu_z=|z|, \ G_z$ is bounded in $z$.
Since
\begin{equation}\label{2.1}
G_z=G_z(A,\mu_z)=
(\mu_zI+A)^{-1}(zI+A),
\end{equation}
the error reduction is now measured by
\begin{equation}\label{2.3}
\|I-\al G_z(A,\mu_z)\|=\max_{\la\in\si(A)}
|1-\al G_z(\la,\mu_z)|,
\quad G_z(\la,\mu_z)=\frac{z+\la}{\mu_z+\la},
\end{equation}
and we want to choose $\al$ so that this quantity is as small
as possible.
% We show the following.

\begin{theorem}\label{th2.3} 
Let $z=x+iy$ with $\arg z\in (-\pi,\pi)$, let $\mu_z>-\la_1$,
and determine $\alpha=\al_z$ by
taking $\a=G_z(\la_1,\mu_z)$ and $\b=G_z(\la_N,\mu_z)$ in Lemma~\ref{l2.1}.  
Then the error reduction
factor  in \eqref{2.2} is bounded independently of $\la_N$ by
\begin{equation}\label{2.wtep}
\wt\ep_z:=\|I-\al_z G_z(A,\mu_z)\|=|1-\al_z\a|\le c(z,\lambda_1,\mu_z)<1.
\end{equation}
\end{theorem}
\begin{proof}
We note that 
\[
G_z(\la,\mu_z)=1+\frac{z-\mu_z}{\mu_z+\la}\in
	\bigl[G_z(\la_1,\mu_z), G_z(\la_N,\mu_z)\bigr], 
	\quad\text{for $\la\in[\la_1,\la_N]$,}
\]
and that $G_z(\la_N,\mu_z)\to 1$ as $\la_N\to\infty$.  Thus,
$G_z(\la,\mu_z)\in[\a,1]$ for all $\la\in\si(A)$, and
since this is a fixed line segment, Lemma \ref{l2.1} shows the
theorem.
\end{proof}

Since, for $z$, $\lambda_1$ and $\lambda_N$ given, the factor~$\wt\ep_z$
is an explicit, albeit complicated, function of~$\mu_z$, it is natural
to choose $\mu_z$ as the value that minimizes this function.  The
numerical values of~$\mu_z$ used in this section were determined in
this way, via an optization routine,
\texttt{scipy.optimize.fminbound}~\cite{scipy}, 
based on a well-known algorithm due to Brent that does not require 
derivative values.  We obtained almost identical results, not shown here,
by setting $b=1$, corresponding to~$\lambda_N=\infty$.

In Table~\ref{tab1}, we see the dramatic effect of the 
preconditioner~$B_z=(\mu_zI+A)^{-1}$ on the
error reduction factor, in the case of the model problem from 
Section~\ref{sec5}, with~$z=z_j$.  
Notice that $\wt\ep_z$ increases with~$j$, whereas $\ep_z$
decreases.

\begin{table}
\renewcommand{\arraystretch}{1.2}
\begin{center}
\caption{Richardson iteration with $\alpha=\rho e^{-i\varphi}$, and
preconditioning with $B_z=(\mu_zI+A)^{-1}$.}
\label{tab1}
\begin{tabular}{rrr|rrr|rrrr}
\multicolumn{3}{c|}{}&
\multicolumn{3}{c|}{Theorem~\ref{th1}}&
\multicolumn{4}{c}{Theorem~\ref{th2.3}}\\
\hline
\multicolumn{1}{c}{$j$}  &
\multicolumn{1}{c}{$x_j$}  &
\multicolumn{1}{c|}{$y_j$}  &
\multicolumn{1}{c}{$\rho_z$}  &
\multicolumn{1}{c}{$\varphi_z$}  &
\multicolumn{1}{c|}{$\ep_z$}  &
\multicolumn{1}{c}{$\rho_z$}  &
\multicolumn{1}{c}{$\varphi_z$}  &
\multicolumn{1}{c}{$\mu_z$}  &
\multicolumn{1}{c}{$\wt\ep_z$}  \\
\hline 
  0&   0.00&   0.00&  4.99e-04&  -0.00& 0.9995&  1.000&  -0.00&   0.00&  0.000\\
  2&  -0.05&   0.30&  4.93e-04&   0.15& 0.9995&  0.988&   0.15&   0.00&  0.152\\
  4&  -0.18&   0.64&  4.73e-04&   0.33& 0.9995&  0.947&   0.33&   0.03&  0.321\\
  6&  -0.43&   1.02&  4.31e-04&   0.53& 0.9996&  0.864&   0.53&   0.16&  0.503\\
  8&  -0.81&   1.51&  3.76e-04&   0.72& 0.9996&  0.753&   0.72&   0.51&  0.658\\
 10&  -1.35&   2.12&  3.24e-04&   0.86& 0.9995&  0.650&   0.86&   1.14&  0.760\\
 12&  -2.10&   2.93&  2.85e-04&   0.96& 0.9995&  0.571&   0.96&   2.11&  0.821\\
 14&  -3.13&   4.01&  2.58e-04&   1.03& 0.9994&  0.516&   1.03&   3.52&  0.856\\
 16&  -4.54&   5.45&  2.39e-04&   1.07& 0.9993&  0.478&   1.07&   5.47&  0.878\\
 18&  -6.45&   7.38&  2.25e-04&   1.10& 0.9991&  0.451&   1.10&   8.15&  0.892\\
 20&  -9.02&   9.97&  2.16e-04&   1.12& 0.9988&  0.432&   1.12&  11.78&  0.902\\
\hline
 21& -20.00&  20.00&  1.98e-04&   1.17& 0.9978&  0.395&   1.17&  26.56&  0.919\\
\end{tabular}
\end{center}
\end{table}

\begin{table}
\renewcommand{\arraystretch}{1.2}
\begin{center}
\caption{Preconditioned Richardson iterations using Theorem~\ref{2t.prec}.}
\label{tab2}
\begin{tabular}{r|rrr|rrr|rrr}
&\multicolumn{6}{c|}{$B_z=(\mu_zI+A)^{-1}$} &
\multicolumn{3}{c}{Incomplete Cholesky}\\
\hline
\multicolumn{1}{c|}{$j$}  &
\multicolumn{1}{c}{$\rho_z$}  &
\multicolumn{1}{c}{$\varphi_z$}  &
\multicolumn{1}{c|}{$\wh\ep_z$}  &
\multicolumn{1}{c}{$\rhosz$}  &
\multicolumn{1}{c}{$\varphisz$}  &
\multicolumn{1}{c|}{$\epsilonsz$} &
\multicolumn{1}{c}{$\rho_z$}  &
\multicolumn{1}{c}{$\varphi_z$}  &
\multicolumn{1}{c}{$\wh\ep_z$}  \\
%\hline 
% 0.00&       &       &       &       &       &       &  1.278&   0.00&  0.987\\
\hline
 0&  1.000&   0.00&  0.000&  1.000&   0.00&  0.000&  0.643&   0.00&  0.997\\
 2&  0.510&   0.55&  0.751&  0.988&   0.15&  0.152&  0.606&   0.23&  0.997\\
 4&  0.335&   0.73&  0.866&  0.947&   0.33&  0.321&  0.554&   0.40&  0.997\\
 6&  0.226&   0.89&  0.926&  0.864&   0.53&  0.503&  0.479&   0.60&  0.998\\
 8&  0.160&   1.03&  0.958&  0.753&   0.72&  0.658&  0.391&   0.79&  0.998\\
10&  0.122&   1.12&  0.973&  0.650&   0.86&  0.760&  0.314&   0.94&  0.998\\
12&  0.100&   1.19&  0.981&  0.571&   0.96&  0.821&  0.256&   1.04&  0.998\\
14&  0.087&   1.23&  0.985&  0.516&   1.03&  0.856&  0.213&   1.11&  0.998\\
16&  0.079&   1.26&  0.988&  0.478&   1.07&  0.878&  0.182&   1.15&  0.998\\
18&  0.073&   1.28&  0.989&  0.451&   1.10&  0.892&  0.157&   1.19&  0.997\\
20&  0.070&   1.29&  0.990&  0.432&   1.12&  0.902&  0.138&   1.22&  0.997\\

%\hline
%21&  0.062&   1.32&  0.992&  0.395&   1.17&  0.919&  0.101&   1.27&  0.997\\
%0.100&   1.27&  0.997\\
\end{tabular}
\end{center}
\end{table}

Since computing the action of~$(\mu_zI+A)^{-1}$ is expensive, we
now want to consider a more general preconditioner~$B_z$ (still
assumed to be positive definite and Hermitian).
Suppose first that $z=0$ and write $B=B_0$.
If $B^{-1}$ is spectrally equivalent to~$A$, that is, if
\begin{equation}
\label{2.ff}
m (B^{-1}v,v)\le (Av,v)\le M(B^{-1}v,v),\quad\forall v\in V,
\end{equation}
for some positive $m$~and $M$, then for suitable~$\alpha$ the iterative 
scheme converges geometrically with respect to a suitable energy norm.
More precisely, setting
\begin{equation}\label{2.bb}
[v,w]:=(B^{-1}v,w),\quad \db{v}:=[v,v]^{1/2},
\end{equation}
the operator~$BA$ is Hermitian with respect to $[\cdot,\cdot]$,
with eigenvalues~$\lambda_j(BA)=\lambda_j(A,B^{-1})$ in the closed
interval~$[m,M]$, so that $\ka(BA)\le M/m$, and thus
\begin{equation}\label{2.eps}
\db{I-\alpha BA}=\frac{\kappa(BA)-1}{\kappa(BA)+1}
\le \frac{M-m}{M+m}\quad\text{if}\quad
\alpha=\frac{2}{\lambda_1(BA)+\lambda_N(BA)}.
\end{equation}
%Note that, by \eqref{2.ff}, $\db{v}$ is equivalent to the energy norm 
%$(Av,v)^{1/2}$.

In the general case of \eqref{2.10} with $z\ne0$, we shall 
write $G_z$ in the form
\begin{equation}
\label{2.Gz1}
G_z=B_zA_z=\wh z\, B_z+B_z(\mu_zI+A),\where \wh z:=z-\mu_z.
\end{equation}
We take $B_z^{-1}$ to be spectrally equivalent to~$\mu_zI+A$, replacing 
the assumption~\eqref{2.ff} by
\begin{equation}\label{2.ff1}
m_z (B_z^{-1}v,v)\le 
((\mu_zI+A)v,v)
\le M_z(B_z^{-1}v,v),\quad\forall v\in V,
\end{equation}
and define the associated inner product and norm,
\begin{equation}
[v,w]:=(B_z^{-1}v,w),\quad \db{v}:=[v,v]^{1/2},
\label{2.bb1}
\end{equation} 
which now depend on~$z$.
The operator $B_z(\mu_zI+A)$ is then Hermitian with respect to
$[\cdot,\cdot]$, with eigenvalues in the closed interval~$[m_z,M_z]$.

In \cite{SST03},  the preconditioning of (2.1) by
using an operator $B$ independent of $z$, corresponding to $\mu_z=0$,
was briefly discussed, and this turned out to be
advantageous only for small~$|z|$. Here
we shall show the following estimate in the present
more general case for the error reduction factor with respect 
to the norm~$|[\cdot]|$,
which is an improvement of the result in \cite{SST03}.
For simplicity we assume $y=\Im z>0$.

\begin{theorem}\label{2t.prec}
Consider the preconditioned equation \eqref{2.Gz1} and the corresponding
iterative scheme \eqref{2.2}. Let $\al_z$ be determined as follows:
With $\wh z=z-\mu_z$,
assume that $\zeta=\arg \wh z\in(\tfrac{1}{2}\pi,\pi)$ and that
$B_z$ satisfies~\eqref{2.ff1}. 
Let $\fy_z=-\arg\al_z$ be the value in $J:=(\zeta-\tfrac12\pi,\tfrac12\pi)$ 
that maximizes the function
\[
\nu_z(\fy):=
\frac{m_z\cos^2\fy\,\cos(\zeta-\fy)}{M_z\cos(\zeta-\fy)+\La_z\cos\fy},
\where   \La_z=|\wh z|\,\|B_z\|,
\] 
and set
$\rho_z=|\al_z|=\nu_z(\fy_z)/(m_z\cos\fy_z)$.
Then we have for the error reduction factor
\begin{equation}\label{2.est0}
\db{I-\alpha_z G_z}\le \wh\ep_{z}:= \big(1-\nu_z(\fy_z)\big)^{1/2},
\quad\al_z:=\rho_z e^{-i\fy_z}.
\end{equation}
If, in addition, there is a $\gamma_z\ge0$ such that
\begin{equation}\label{2.ga}
\Re\big(\wh z\,[B_zv,B_z(\mu_zI+A)v]\big)\le
-\ga_z
[B_zv,v],\quad\forall v\in V,
\end{equation}
we  define $\alphasz$ by choosing $\varphisz=-\arg\alphasz\in J$
to maximize the function
\begin{equation}\label{2.wt}
\nusz(\varphi):=
\frac{m_z\cos^2\fy\,\cos(\zeta-\fy)}
{\max(M_z\cos(\zeta-\fy),\Lambdasz\cos\fy)},
\where \Lambdasz:=\La_z-\frac{2\ga_z}{|\wh z|}, 
\end{equation}
and put $\rhosz=|\alphasz|=\nusz(\varphisz)/(m_z\cos\varphisz)$. 
We then have the sharper estimate
\begin{equation}\label{2.est1}
\db{I-\alphasz G_z}\le \epsilonsz:=
	\big(1-\nusz(\varphisz)\big)^{1/2},
\quad \alphasz:=\rhosz e^{-i\varphisz}.
\end{equation}
\end{theorem}
\begin{proof}
We have, for $\al=\rho\,e^{-i\fy}$,
\begin{align}\label{opt}
\db{(I-\al G_z)v}^2
=\db{v}^2-
2\Re(\alpha[G_zv,v])
+|\al|^2\db{G_zv}^2,
\end{align}
so, writing for brevity
$c_0:=c_0(\fy)=\cos\fy$ and $c_1:=c_1(\fy)=\cos(\zeta-\fy)$,
and noting that $\wh z=|\wh z|\,e^{i\zeta}$, 
\begin{equation}\label{eq: Re(alpha[Gv,v])}
\begin{aligned}
\Re\,(\alpha[G_zv,v])
&=
\Re\alpha\,[B_z(\mu_zI+A)v,v]+
\Re(\alpha \wh z)[B_zv,v]
\\&=
\rho\big(
c_0[B_z(\mu_zI+A)v,v]+
c_1|\wh z| [B_zv,v]\big).
\end{aligned}
\end{equation}
Noting that $c_0>0$ and $c_1>0$ for $\fy\in J$, we find,
since $|\wh z|\,\|B_z\|=\La_z,$
\begin{align}
\label{2.waste}
\db{G_zv}&\le \db{B_z(\mu_zI+A)v}+|\wh z|\,\db{B_zv}
\\
&\le M_z^{1/2} [B_z(\mu_zI+A)v,v]^{1/2}+\La_z^{1/2}|\wh z|^{1/2}[B_zv,v]^{1/2} 
\notag
\\
&\le(c_0^{-1}M_z+c_1^{-1}\La_z)^{1/2}
\big(c_0[B_z(\mu_zI+A)v,v]+c_1|\wh z| [B_zv,v]\big)^{1/2},
\notag
\end{align}
so that, by \eqref{opt}~and \eqref{eq: Re(alpha[Gv,v])}
\begin{align*}
\db{(I-\al G_z)v}^2
&\le
\db{v}^2-
2\rho\big(
c_0[B_z(\mu_zI+A)v,v]+
c_1|\wh z| [B_zv,v]\big)
\\
&\quad
+\rho^2(c_0^{-1}M_z+c_1^{-1}\La_z)
\big(
c_0[B_z(\mu_zI+A)v,v]+
c_1|\wh z| [B_zv,v]\big).
\end{align*}
Minimizing in $\rho$ we find
$\rho=1/(c_0^{-1}M_z+c_1^{-1}\La_z)
=c_0c_1/(c_1M_z+c_0\La_z),$
 and hence
 \[
 \db{(I-\alpha G_zv}^2
 \le \db{v}^2-\frac{c_0c_1}{c_1M_z+c_0\La_z}
 \big(c_0[B_z(\mu_zI+A)v,v]+
 c_1|\wh z| [B_zv,v]\big).
 \]
Here, by \eqref{2.ff1},
 \[
 c_0[B_z(\mu_zI+A)v,v]+
 c_1|\wh z| [B_zv,v]
\ge 
 c_0[B_z(\mu_zI+A)v,v]
\ge 
m_zc_0\db{v}^2,
 \]
and thus, remembering that $\rho$ depends on~$\varphi$ through $c_0$~and $c_1$,
 \[
 \db{(I-\alpha G_z)v}^2
 \le \db{v}^2
 -\frac{c_0^2\,c_1\,m_z}{c_1M_z+c_0\La_z}
 \db{v}^2=(1-\nu_0(\fy))
\db{v}^2.
 \]
Minimizing in $\fy$ over $J$ shows the result stated.

The first inequality in \eqref{2.waste} could be somewhat wasteful.
If we assume  that \eqref{2.ga} holds, then we find, instead
of~\eqref{2.waste},
\begin{align*}
\db{G_zv}^2&= \db{B_z(\mu_zI+A)v}^2+|\wh z|^2\,\db{B_zv}^2
+2\Re\big(\wh z\,[B_zv,B_z(\mu_zI+A)v]\big)
\\
&\le \db{B_z(\mu_zI+A)v}^2+|\wh z|^2\,\db{B_zv}^2
 -2\,\ga_z[B_zv,v]
\\
&\le M_z [B_z(\mu_zI+A)v,v]+|\wh z|(\La_z-2\ga_z/|\wh z|)[B_zv,v] 
\\
&\le\max(c_0^{-1}M_z,c_1^{-1}\Lambdasz)
\big(c_0[B_z(\mu_zI+A)v,v]+c_1|\wh z| [B_zv,v]\big),
\end{align*}
so that, by \eqref{opt},
\begin{align*}
\db{(I-&\al G_z)v}^2 \le
\db{v}^2- 2\rho\big( c_0[B_z(\mu_zI+A)v,v]+
c_1|\wh z| [B_zv,v]\big)
\\
&+\rho^2\max(c_0^{-1}M_z,c_1^{-1}\Lambdasz)
\big(
c_0[B_z(\mu_zI+A)v,v]+
c_1|\wh z| [B_zv,v]\big).
\end{align*}
The proof of~\eqref{2.est1} is now finished in the same way
as that of~\eqref{2.est0} above.
\end{proof}

In the limiting case when~$z\to0$~and $\mu_z\to0$, with~$\zeta\to\tfrac12\pi$,
the method of analysis in Theorem~\ref{2t.prec} gives
\begin{equation}\label{eq: gen precond z=0}
\alpha_0=\alphaso=\frac{1}{M}\quad\text{and}\quad
\wh\ep_0=\epsilonso=\sqrt{1-\frac{m}{M}}
	\approx1-\frac{m}{2M},
\end{equation}
compared to the error reduction ratio~$(M-m)/(M+m)\approx1-2m/M$
in~\eqref{2.eps}.

Applying Theorem~\ref{2t.prec} in the special case
$B_z=(\mu_zI+A)^{-1}$, with
\[
\|B_z\|=\frac{1}{\la_1+\mu_z},\quad m_z=M_z=1,\quad \ga_z=-\Re\wh z,
\]
we see from Table~\ref{tab2} that, for our model problem, 
$\rhosz$, $\varphisz$ and $\epsilonsz$
are close to the corresponding values in Table~\ref{tab1},
but the values of~$\wh\ep_{z}$ are worse.  (The points~$z_j$ are the
same for both tables, as are the values of~$\mu_z$.)

\begin{table}
\renewcommand{\arraystretch}{1.2}
\begin{center}
\caption{Richardson iteration preconditioned by $k$~V-cycles of AMG.}
\label{tab: Richardson AMG}
\begin{tabular}{r|ccc|cc|ccc}
\multicolumn{1}{c|}{$j$} &
$\rho_z$  &
$\varphi_z$  &
$\wh\ep_z$  &
$k$ &
$\lambda_1(F_z)$ &
$\rhosz$  &
$\varphisz$  &
$\epsilonsz$ \\
\hline
  0&  1.000&  0.00& 0.643&   1&   0.000& 1.000&  0.00& 0.643\\
  2&  0.517&  0.54& 0.767&   3&   0.003& 0.919&  0.41& 0.464\\
  4&  0.341&  0.73& 0.873&   3&   0.086& 0.811&  0.63& 0.623\\
  6&  0.238&  0.88& 0.934&   2&   0.043& 0.541&  1.00& 0.869\\
  8&  0.166&  1.02& 0.962&   2&   0.301& 0.429&  1.13& 0.918\\
 10&  0.125&  1.12& 0.976&   2&   0.738& 0.341&  1.22& 0.947\\
 12&  0.102&  1.19& 0.983&   2&   1.351& 0.277&  1.27& 0.963\\
 14&  0.093&  1.23& 0.989&   1&   0.387& 0.185&  1.30& 0.983\\
 16&  0.083&  1.26& 0.991&   1&   1.112& 0.175&  1.32& 0.985\\
 18&  0.077&  1.28& 0.992&   1&   2.388& 0.169&  1.34& 0.985\\
 20&  0.072&  1.29& 0.992&   1&   3.426& 0.158&  1.35& 0.987\\
\end{tabular}
\end{center}
\end{table}

To better understand the condition~\eqref{2.ga} for general~$B_z$, we
write
\begin{equation}\label{eq: Hz Fz}
H_z:=B_z(\mu_zI+A)=H^+_z+iH^-_z
\quad\text{and}\quad
F_z:=|x-\mu_z|H^+_z-yH^-_z,
\end{equation}
where the Hermitian operators~$H^\pm_z$ are defined by
\[
H^+_z:=\tfrac12(H_z+H_z^*)\quad\text{and}\quad 
H^-_z:=-i\tfrac12(H_z-H_z^*),
\]
and we have used ${}^*$ to denote the adjoint with 
respect to~$(\cdot,\cdot)$.
When $B_z$ commutes with~$A$, the operator~$H_z$ is Hermitian in~$V$,
and \eqref{2.ga} follows if $\ga_z\le |\Re\wh z|\,m_z$, because
$\la_1(H_z)\ge m_z$ by our assumption~\eqref{2.ff1}.
This result is contained as the case~$H^-_z=0$ 
of the following proposition.

\begin{proposition}\label{prop: gamma_z}
Fix $\wh z=\wh x+i\, y$ with $\wh x=x-\mu_z\le 0$
and $y\ge 0$, and let $F_z$ be the Hermitian operator defined 
in~\eqref{eq: Hz Fz}.
Then a necessary and sufficient condition for~\eqref{2.ga} is that
$0\le\gamma_z\le\lambda_1(F_z)$.
\end{proposition}
\begin{proof}
We find that
$[B_zv,B_z(\mu_zI+A)v]=(v,H_zv)=(v,H^+_zv)-i\,(v,H^-_zv)$, so
\[
-\Re\bigl(\wh z[B_zv,B_z(\mu_zI+A)v]\bigr)
	=|\wh x|(v,H^+_zv)- y\,(v,H^-_zv)=(v,F_zv).
\]
Since $F_z$ is Hermitian and $[B_zv,v]=(v,v)$, it follows that
\[
\inf_{0\ne v\in V}\frac{-\Re\bigl(\wh z[B_zv,B_z(\mu_zI+A)v]\bigr)}{[B_zv,v]}
	=\inf_{0\ne v\in V}\frac{(v,F_zv)}{(v,v)}=\lambda_1(F_z).
\]
\end{proof}

In the Hermitian case, $H_z^-=0$, this proposition implies that \eqref{2.est1}
holds with $\Lambdasz=\La_z-2m_z|\cos\zeta|$ in~\eqref{2.wt}.
In general, since
\[
\lambda_1(F_z)\ge|\hat x|\lambda_1(H_z^+)-y\|H_z^-\|,
\]
a sufficient condition for $\lambda_1(F_z)\ge0$ is that
$\|H_z^-\|\le|\hat x|y^{-1}\lambda_1(H_z^+)$, which makes
$H_z$ essentially Hermitian.
%In general,  the condition
%$\wt m_z|\wh x|-\wt M_z y\ge0$  requires
%$\wt M_z=\|H_z^-\|$ to be of the same order as
%$\wt m_z=\la_1(H_z^+)$, or that $H_z$ be essentially Hermitian.  

We  have also the following simple consequence of 
Proposition~\ref{prop: gamma_z}.

\begin{corollary}\label{cor: ||Hz-I||}
If $\|H_z-I\|\le\delta$ for some $\delta<|\wh x|/(|\wh x|+y)$, then
\eqref{2.ga} is satisfied with $\gamma_z=|\wh x|-\delta(|\wh x|+y)$.
\end{corollary}
\begin{proof}
Since $\|H_z^*-I\|=\|H_z-I\|$, we have
$\|H_z^-\|=\|\tfrac12(H_z-I)-\tfrac12(H_z^*-I)\|\le\delta$ and
$\|H_z^+-I\|\le\tfrac12\|H_z-I\|+\tfrac12\|H_z^*-I\|=\|H_z-I\|\le\delta$,
so it follows from $F_z=|\wh x|I+|\wh x|(H_z^+-I)-yH_z^-$ that
\[
(v,F_zv)\ge|\wh x|\|v\|^2-\delta|\wh x|\|v\|^2-\delta y\|v\|^2=\gamma_z\|v\|^2
\quad\text{for all $v\in V$.}
\]
Hence, $\lambda_1(F_z)\ge\gamma_z>0$.
\end{proof}

We now consider the practical application of these methods to the
linear system~\eqref{1.matsg}.  Putting $\A_z:=z\M +\S$ so that
$\A_z\vecw=\vecg$, the basic Richardson iteration~\eqref{6.5}
takes the form
\[
\vecw^{n+1}=(\I-\alpha\M^{-1}\A_z)\vecw^n+\alpha\M^{-1}\vecg
	=\vecw^n+\alpha\M^{-1}\vecr^n,
\]
where $\vecr^n:=\vecg-(z\M+\S)\vecw^n$ denotes the $n$th residual.
For the lumped mass method, we replace~$\M$ throughout by the corresponding
diagonal matrix~$\D$, whose inverse is trivial to compute.

In the case of the special preconditioner $B_z=(\mu_z I+A)^{-1}$,
we find that $G_z\vecv=(\mu_zI+A)^{-1}A_z\vecv=(\mu_z\M+S)^{-1}\A_z\vecv$ 
and so \eqref{2.2} takes the form
\[
\vecw^{n+1}=\vecw^n+\alpha(\mu_z\M+\S)^{-1}\vecr^n.
\]
We may write a general preconditioner in the form~$B_z\vecv=\B_z\M\vecv$, 
where $\B_z$ is Hermitian and positive-definite with respect to the standard
unitary inner product on~$\C^N$, since then
$(B_z\vecv,\vecw)=\iprod{\B_z\M\vecv,\M\vecw}$.  In this way,
\[
\vecw^{n+1}=\vecw^n+\alpha\B_z\vecr^n.
\]
The condition~\eqref{2.ff1} is equivalent to
\[
m_z\iprod{\B_z^{-1}\vecu,\vecu}\le\iprod{(\mu_z\M+\S)\vecu,\vecu}
	\le M_z\iprod{\B_z^{-1}\vecu,\vecu}
	\quad\forall \vecu\in\C^N,
\]
which means that 
$\lambda_j(\mu_z\M+\S,\B_z^{-1})$ belongs to the closed 
interval~$[m_z,M_z]$ for all~$j$.  In Tables \ref{tab2}~and 
\ref{tab: Richardson AMG}, the values of~$\mu_z$ are the same as in
Table~\ref{tab1}, and for our computations we used
best possible values $m_z=\lambda_1(\mu_z\M+\S,\B_z^{-1})$
and $M_z=\lambda_N(\mu_z\M+\S,\B_z^{-1})$.
Note also that $\|B_z\|=\lambda_N(\M,\B_z^{-1})=\lambda_N(\B_z,\M^{-1})$ 
because $B_z\vecv=\lambda\vecv$ is equivalent to 
$\M\vecv=\lambda\B_z^{-1}\vecv$ and to $\B_z(\M\vecv)=\lambda\M^{-1}(\M\vecv)$.

To apply Proposition~\ref{prop: gamma_z}, we introduce Hermitian 
matrices
\[
\H^+_z:=\mu_z\M\B_z\M+\tfrac12(\S\B_z\M+\M\B_z\S)
\quad\text{and}\quad
\H^-_z:=i\tfrac12(\S\B_z\M-\M\B_z\S),
\]
so that $\H^\pm_z\vecv=\M H^\pm_z\vecv$ for all~$\vecv\in\C^N$, and then
put $\F_z:=|x-\mu_z|\H^+_z-y\H^-_z$ so that
$F_z\vecv=\M^{-1}\F_z\vecv$.  In this way,
$\lambda_1(F_z)=\lambda_1(\F_z,\M)$.

Table~\ref{tab2} also shows the values of $\rho_z$, $\fy_z$ 
and $\wh\ep_{z}$ using $\B_z=(\L_z\L_z^T)^{-1}$ for an incomplete
Cholesky factorization~$\L_z\L_z^T\approx\mu_z\I+\A$, computed using
\cite{JonesPlassmann1995}.  
Although better than than no preconditioning, the error reduction factors
are still too close to unity for the method to be of practical use.  
We can compare the values when~$z=0$ to the optimal ones given 
by~\eqref{2.eps}.  In our case,
$\lambda_1(\B_0\S)=0.0102$~and $\lambda_N(\B_0\S)=1.55$, so
$\alpha=1.28$, $\kappa=152.0$ and $(\kappa-1)/(\kappa+1)=0.987$,
compared to
the values $\alpha_0=\alphaso=0.643$~and 
$\wh\ep_{0}=\epsilonso=0.997$
given by~\eqref{eq: gen precond z=0}.  For~$j\ge1$, we found that 
$\lambda_1(F_z)<0$ at $z=z_j$, so we could not apply the second 
estimate~\eqref{2.est1} of Theorem~\ref{2t.prec}.

To find a better preconditioner, consider any symmetric, linear iterative
process for the equation~$(\mu_z\M+\S)\vecv=\vecg$, of the form
\begin{equation}\label{eq: lin it}
\vecv^{j+1}=\vecv^j+\B_z\bigl(\vecg-(\mu_z\M+\S)\vecv^j\bigr),
\quad\text{with $\B_z^T=\B_z$.}
\end{equation}
Performing $k$~steps of this iteration defines another linear iterative
process,
\begin{equation}\label{eq: k lin it}
\vecv^{j+k}=\vecv^j+\B_{z,k}\bigl(\vecg-(\mu_z\M+\S)\vecv^j\bigr),
\end{equation}
and the relation between $\B_z=\B_{z,1}$ and $\B_{z,k}$ may be seen 
from the error reduction operator:
\[
\I-\B_{z,k}(\mu_z\M+\S)=\bigl(\I-\B_z(\mu_z\M+\S)\bigr)^k.
\]
It follows that $\B_{z,k}^T=\B_{z,k}$, so the $k$-step process is also
symmetric.  The 1-step process converges if and only if 
$\sigma\bigl(\I-\B_z(\mu_z\M+\S)\bigr)\subseteq[-\varrho_z,\varrho_z]$
for some $\varrho_z<1$, because $\B_z(\mu_z\M+\S)$ is symmetric with respect 
to the inner product~$\iprod{(\mu_z\M+\S)\vecv,\vecw}$; 
cf.~Bramble~\cite[page~4]{Br}.
In this case, the eigenvalues of $\B_z(\mu_z\M+\S)$ lie in the 
interval~$[1-\varrho_z,1+\varrho_z]$, or equivalently,
\[
(1-\varrho_z)\iprod{(\mu_z\M+S)^{-1}\vecv,\vecv}\le\iprod{\B_z\vecv,\vecv}
	\le(1+\varrho_z)\iprod{(\mu_z\M+S)^{-1}\vecv,\vecv}
\]
for all $\vecv\in\C^N$, showing that $\B_z$ is positive-definite.  
In the same way, the eigenvalues of~$\B_{z,k}$ lie in the interval
$[(1-\varrho_z)^k,(1+\varrho_z)^k]$ and $\B_{z,k}$ is positive-definite.
Thus, any symmetric and
convergent linear iterative process yields a suitable 
preconditioner~$\B_{z,k}$, and moreover the hypothesis of
Corollary~\ref{cor: ||Hz-I||} will be satisfied for~$k$ sufficiently large,
because $\H_z=\B_{z,k}(\mu_z\M+\S)\to\I$ as~$k\to\infty$.

Table~\ref{tab: Richardson AMG} shows the results obtained when one step
of the linear iteration~\eqref{eq: lin it} corresponds to a single V-cycle
of a symmetric, algebraic multigrid (AMG) solver~\cite{BellOlsonSchroder2011}, 
and thus \eqref{eq: k lin it} corresponds to $k$~V-cycles.
For each quadrature point~$z_j$, the value of~$k$ shown is the 
smallest for which $\lambda_1(F_z)\ge0$, allowing application of
Proposition~\ref{prop: gamma_z}.

The need to compute $m_z$~and $M_z$, and ideally also %$\wt m_z$~and $\wt M_z$
$\lambda_1(F_z)$, to determine a good choice of the acceleration 
parameter~$\alpha$, means that Richardson iteration is less attractive
in practice than the Krylov methods of the next section, which do not 
suffer from this drawback, and also exhibit faster convergence.
%%%%%%%%%%%%%%%%%%%%%%%%%%%%%%%%%%%%%%%%%%%%%%%%%%%%%%%%%%%%%%%%%%
\section{Conjugate gradient method}
\label{sec3}

Once again, assume that $A$ is a positive definite Hermitian operator in a
finite-dimensional complex inner product space $V$, 
%with eigenvalues $0<\lambda_1\le\lambda_2\le\cdots\le\lambda_N$,
and consider the equation
\begin{equation}\label{cg1}
A_zw=g,\where A_z:=zI+A,\ z=x+iy,\ \arg z\in (-\pi,\pi).
\end{equation}
Given $w_0$, a preliminary guess for the solution $w$, 
we define the residual~$r_0:=g-A_zw_0$ and the associated 
Krylov subspace of order~$n\ge1$,
\[
V_n:=\vecspan\{r_0, A_zr_0, \dots, A_z^{n-1} r_0\}=
\vecspan\{r_0, Ar_0, \dots, A^{n-1} r_0\},
\] 
with $V_0:=\{0\}$.  Note that $V_n$ depends on $z$ through $r_0$. 
The exact soloution of \eqref{cg1} satisfies
\begin{equation}
\label{3.ex}
(A_zw,\fy)=(g,\fy),\quad\forall\ \fy\in V.
\end{equation}
As in the classical conjugate gradient method, we define
the approximate solution  
%\eqref{cg1} 
$w_n=w_0+v_n$, with $v_n\in  V_n$,
by Galerkin's method,  or
\begin{equation}\label{cg2}
(A_zw_n,\fy)
=(g,\fy),\quad\forall\fy\in V_n,
\end{equation}
and find that $v_n=w_n-w_0$ satisfies
\begin{align*}
(A_zv_n,\fy)
=(A_z(w_n-w_0),\fy)
=(g,\fy)-
(A_zw_0,\fy)=(r_0,\fy),\quad \forall\fy\in V_n.
\end{align*}
The solution of \eqref{cg2} is therefore unique,
because if $r_0=0$ we have
\[
(A_zv_n,v_n)=
z\|v_n\|^2+
(Av_n,v_n)=0,
\]
which implies $v_n=0$. Hence there also exists a solution of 
the finite dimensional problem \eqref{cg2}. 
The error
$e_n:=w_n-w$ satisfies
\begin{equation}\label{4.erreq}
(A_ze_n,\fy)=0,\quad\forall\fy\in V_n.
\end{equation}
	
To study the convergence of $w_n$, we introduce the norm 
\begin{equation}\label{3.tbdef}
\tb{v}^2:=
|z|\|v\|^2
+(Av,v),
\end{equation}
and note the following lemma.

\begin{lemma}\label{l4.1}
If $\arg z=\phi\in(-\pi,\pi)$, then for all $v$, $w\in V$ we have
\[
|(A_zv,w)|\le\tb{v}\,\tb{w}
\andy\
|(A_zv,v)|\ge
\cos(\tfrac12\phi)\,\tb{v}^2.
\]
\end{lemma}
\begin{proof}
The first part follows at once from
\[
|(A_zv,w)|\le |z|\,|(v,w)| +|(Av,w)|
	\le |z|\,\|v\|\,\|w\| +(Av,v)^{1/2} \,(Aw,w)^{1/2}.
\]
Setting $\be:=e^{-i\phi/2}$, the second part now results from
\begin{align*}
\Re\bigl(\be(A_zv,v)\bigr)&= \Re(\be z)\|v\|^2 +\Re\be (Av,v) \\
&\ge|z|\cos(\tfrac12\phi)\|v\|^2 +\cos(\tfrac12\phi)\,(Av,v)
	=\cos(\tfrac12\phi)\tb{v}^2.
\end{align*}
\end{proof}

Using this lemma, we have the following quasi-optimality result.

\begin{proposition}\label{l4.2}
Let $w$ and $w_n$ be the solutions of \eqref{cg1}~and \eqref{cg2}, respectively.
Then, for $\arg z=\phi\in(-\pi,\pi)$, 
\[
\tb{w_n-w}\le\sec(\tfrac12\phi)\, \inf_{v\in w_0+V_n}\tb{v-w}.
\]
\end{proposition}
\begin{proof}
Lemma~\ref{l4.1} and \eqref{4.erreq} show that, for any $v\in w_0+V_n$,
\begin{align*}
\cos(\tfrac12\phi)\tb{w_n-w}^2 &\le 
	\bigl|\bigl(A_z(w_n-w),w_n-w\bigr)\bigr| 
	=\bigl|\bigl(A_z(w_n-w),v-w\bigr)\bigr| \\
	&\le\tb{w_n-w}\,\tb{v-w},
\end{align*}
which implies the result stated.
\end{proof}

We now proceed to generalize the classical convergence analysis of the 
CG method by allowing for the complex shift in~$A_z$.  Let $\Poly_n$ denote
the space of polynomials of degree at most~$n$, with complex
coefficients.

\begin{theorem}\label{t4.1}
Let $w$ and $w_n$ be the solutions of \eqref{cg1}~and \eqref{cg2}, respectively.
If $Q_n\in\Poly_n$ and $Q_n(0)=1$, then,
for $\arg z=\phi\in(-\pi,\pi)$, 
\[
\tb{e_n}
\le\sec(\tfrac{1}{2}\phi)\,
\max_{\la\in\si(A)}|Q_n(z+\la)|\,\tb{e_0},\where e_n:=w_n-w.
\]
\end{theorem}
\begin{proof}
Let $v:=w+Q_n(A_z)e_0$.  Since $Q_n(\lambda)=1+\lambda P_{n-1}(\lambda)$ 
with $P_{n-1}\in\Poly_{n-1}$ and $r_0=g-A_zw_0=-A_z(w_0-w)
=-A_ze_0$, we have $Q_n(A_z)e_0=e_0-P_{n-1}(A_z)r_0$. 
Hence $v=w_0-P_{n-1}(A_z)r_0\in w_0+V_n$, and we conclude by
Proposition~\ref{l4.2} that
\[
\cos(\tfrac{1}{2}\phi)\tb{e_n}\le\tb{v-w}=\tb{Q_n(A_z)e_0}.
\]
Since $A_z$ is a normal operator,
\[
\|Q_n(A_z)e_0\|\le\max_{\lambda\in\sigma(A)}|Q_n(z+\lambda)|\,\|e_0\|.
\]
Similarly,
\[
\bigl(AQ_n(A_z)e_0, Q_n(A_z)e_0\bigr)
	 \le\max_{\la\in\si(A)}|Q_n(z+\la)|^2\,(Ae_0,e_0),
\]
and we conclude that
\begin{equation*}\label{4.q}
\tb{Q_n(A_z)e_0}\le 
\max_{\la\in\si(A)}|Q_n(z+\la)|\,\tb{e_0},
\end{equation*}
which completes the proof.
\end{proof}

We now introduce the Tchebyshev polynomial $T_n\in\Poly_n$ defined by
\[
T_n(\cos\theta)=\cos(n\theta)\quad\text{for $\theta\in\mathbb C$,}
\]
or, equivalently, since $\cos(i\theta)=\cosh\theta$, by
$T_n(\cosh\theta)=\cosh(n\theta)$,
and show the following consequence of Theorem \ref{t4.1}.

\begin{theorem}\label{t3.2}
With the above notation, we have, for $\phi\in(-\pi,\pi)$,
\[
\tb{e_n} \le\sec(\tfrac{1}{2}\phi)\, |T_n(s_z)|^{-1} \tb{e_0},
\where s_z:=-\frac{\la_1+\la_N+2z}{\la_N-\la_1}.
\]
With $\arg\sqrt{\la_j+z}\in(-\tfrac12\pi,\tfrac12\pi),\
j=1,N,$ 
we  may write
\[
T_n(s_z)=\tfrac12(\eta_z^n+\eta_z^{-n}),
\where \ \eta_z:=-
\frac{\sqrt{\la_N+z}-\sqrt{\la_1+z}}
{\sqrt{\la_N+z}+\sqrt{\la_1+z}}.
\]
Furthermore, $|\eta_z|\le 1-c\lambda_N^{-1/2}$ with~$c=c(z,\lambda_1)>0$.
\end{theorem}
\begin{proof}
The linear change of variables
$s\to\tau$ in the complex plane,
\[
\tau=\tfrac12\big((1-s)(\la_1+z)+(1+s)(\la_N+z)\big),
\]
takes the real interval $[-1,1]$
 onto the segment $[\la_1+z,\la_N+z]$, parallel to the real axis.
We note that
$\tau=0$ when $s=s_z$,
 so that,
if we define
\[
Q_n(\tau):=\frac{T_n(s)}{T_n(s_z)},
\with s=-\frac{\la_1+\la_N+2(z-\tau)}{\la_N-\la_1},
\]
then $Q_n(\tau)\in\mathbb P_n$ and $Q_n(0)=1$. We thus  have
\[
\max_{\la\in[\la_1,\la_N]}|Q_n(\la+z)|=
\max_{\tau\in[\la_1+z,\la_N+z]}|Q_n(\tau)|=
\max_{-1\le s\le1}
\frac{|T_n(s)|}{|T_n(s_z)|}=\frac1{|T_n(s_z)|},
\]
and hence the first statement of the theorem follows by Theorem \ref{t4.1}.

Defining $\theta$ by $\cosh\theta=\tfrac12(e^{\theta}+e^{-\theta})=s_z$
and letting $\eta_z=e^\theta$, we have
\[
T_n(s_z)=T_n(\cosh\theta)=\cosh(n\theta)=\tfrac12(\eta_z^n+\eta_z^{-n}).
\]
Here $\eta_z$ satisfies the quadratic equation  
$\eta_z+\eta_z^{-1}=2s_z$, with roots 
\[
\eta_{z,\pm}(s_z)=
s_z\mp\sqrt{s_z^2-1}=-\tfrac12\big(\sqrt{-s_z+1}\pm\sqrt{-s_z-1}\big)^2.
\]
Setting $\eta_z=\eta_{z,-}(s_z)$, we find
\begin{align*}
\eta_z=-\frac{\sqrt{-s_z+1}-\sqrt{-s_z-1}}{\sqrt{-s_z+1}+\sqrt{-s_z-1}},
\end{align*}
and the stated formula for~$\eta_z$ follows because
$-s_z+1=2(\la_N+z)/(\la_N-\la_1)$ and 
$-s_z-1=2(\la_1+z)/(\la_N-\la_1)$. Furthermore, writing
$\sqrt{-s_z\pm1}=a_\pm+ib_\pm$ we have $a_\pm>0$ with the sign of~$b_+$
the same as that of~$b_-$.  Thus,
\[
|\eta_z|^2=\frac{(a_+-a_-)^2+(b_+-b_-)^2}{
(a_++a_-)^2+(b_++b_-)^2}<1,
\]
and to complete the proof we put
$\kappa_z:=(\lambda_N+z)/(\lambda_1+z)=O(\lambda_N)$ and use
\[
\eta_z=\frac{\sqrt{\kappa_z}-1}{\sqrt{\kappa_z}+1}=
        \frac{1-\kappa_z^{-1/2}}{1+\kappa_z^{-1/2}}=1-2\kappa_z^{-1/2}
        +O(\kappa_z^{-1}).
\]
\end{proof}

\begin{table}
\renewcommand{\arraystretch}{1.2}
%{\tt
\begin{center}
\caption{Error reduction by CG iteration.}
\label{tab3}
\begin{tabular}{rrr|rrr|rr} 
\multicolumn{1}{c}{$j$}& 
\multicolumn{1}{c}{$x_j$}& 
\multicolumn{1}{c|}{$y_j$}& 
\multicolumn{1}{c}{$|\eta_z|$}& 
\multicolumn{1}{c}{$|\wt \eta_z|$}& 
\multicolumn{1}{c|}{$\mu_z$}&
\multicolumn{1}{c}{$|\wt \eta_z|$}& 
\multicolumn{1}{c}{$\mu_z$}\\
\hline
  0&   0.00&   0.00&  0.9687&  0.0000&  0.000&  0.0000&   0.00\\
  2&  -0.05&   0.30&  0.9690&  0.0762&  0.002&  0.0762&   0.00\\
  4&  -0.18&   0.64&  0.9699&  0.1650&  0.031&  0.1652&   0.00\\
  6&  -0.43&   1.02&  0.9708&  0.2698&  0.165&  0.2724&   0.00\\
  8&  -0.81&   1.51&  0.9711&  0.3749&  0.507&  0.3880&   0.00\\
 10&  -1.35&   2.12&  0.9703&  0.4605&  1.138&  0.4948&   0.00\\
 12&  -2.10&   2.93&  0.9686&  0.5221&  2.119&  0.5839&   0.00\\
 14&  -3.13&   4.01&  0.9659&  0.5646&  3.530&  0.6553&   0.00\\
 16&  -4.54&   5.45&  0.9622&  0.5939&  5.492&  0.7121&   0.00\\
 18&  -6.45&   7.38&  0.9577&  0.6143&  8.183&  0.7577&   0.00\\
 20&  -9.02&   9.97&  0.9523&  0.6287& 11.850&  0.7946&   0.00\\

\hline
   & -20.00&  20.00&  0.9364&  0.6570& 26.894&  0.8628&   0.00\\
\end{tabular}
\end{center}
%} % end \tt
\end{table}

Since $|\eta_z|<1$, it follows that $|T_n(s_z)|^{-1}\approx 2|\eta_z|^n$,
and so Theorem~\ref{t3.2} shows linear convergence 
with approximately this rate.  When $A=L_h$, so that 
$\lambda_N\approx ch^{-2}$, the error bound is thus of order $(1-ch)^n$.
The values of~$|\eta_z|$ 
shown in Table~\ref{tab3} refer to the model problem from Section~\ref{sec5},
for which $\lambda_1\approx1$ and $\lambda_N\approx4,000$.  Comparing
the $|\eta_z|$ with the corresponding values of~$\ep_z$ in 
Table~\ref{tab1} confirms the superiority of the CG method over the
Richardson iteration (without preconditioning).

\medskip

We now seek to precondition the CG
method applied to~\eqref{cg1}, and consider first the special 
preconditioner $B_z=(\mu_zI+A)^{-1}$.
We multiply \eqref{cg1} by
$\wt z:=(z-\mu_z)^{-1}$ and $B_z$
to write the equation in the form
\begin{equation}
\label{3.8}
\wt zw+B_zw=\wt z\,B_zg,
\end{equation}
in which thus $\wt z$ and $ B_z$ 
play the roles previously taken by
$z$~and $A$. In particular,
the Krylov subspaces are now
\begin{equation}\label{3.vnt}
V_n=\vecspan\{ r_0,B_z r_0,\dots,B_z^{n-1} r_0\},
\with  r_0=\wt z\,B_z g-(\wt zI+B_z)w_0,
\end{equation}
and the iterates are defined by
\begin{equation}\label{3.teq}
\bigl((\wt zI+B_z) w_n,\fy\bigr)=(\wt z\,B_zg,\fy),
	\quad\forall \fy\in  V_n,\quad w_n=w_0+ v_n,\  v_n\in V_n.
\end{equation}
The  earlier analysis  remains valid, with~$s_z$ now replaced by
\[
\wt s_z:=-\frac{\wt\la_1+ \wt\la_N+ 2\wt z}{\wt \la_N-\wt\la_1},
\with \wt \la_j:=(\mu_z+\la_{N+1-j})^{-1},\ j=1,\,N,
\]
and correspondingly for $\eta_z$. Theorem~\ref{t3.2} then shows that the error
reduction factor is bounded away from~$1$, independently of~$\lambda_N$.

\begin{theorem}\label{2}
For the CG method~\eqref{3.teq} applied to equation~\eqref{3.8}, and
for the norm $\tb{v}^2= | \wt z|\,\|v\|^2 + (B_zv,v)$,
we have
\[ \tb{e_n} \le \sec(\tfrac12\phi)\, |T_n(\wt s_z)|^{-1}\, \tb{e_0}, 
\quad\text{with}\quad T_n(\wt s_z)=\tfrac12(\wt \eta_z^n+\wt \eta_z^{-n}),
\]
where
\begin{equation}\label{3.wt} 
\wt \eta_z:=-\frac{\sqrt{\wt\la_N+\wt z}-\sqrt{\wt\la_1+\wt z}}%
{\sqrt{\wt\la_N+\wt z}+\sqrt{\wt\la_1+\wt z}}
\quad\text{and}\quad |\wt \eta_z|\le c(z,\lambda_1,\mu_z)<1.
\end{equation}
\end{theorem}

We want to discuss how to choose $\mu_z$ to minimize $|\wt \eta_z|$
for a given~$z$.
In practice we are only interested in $z=z_j$ 
with~$\Re z_j\ge\Re z_q\approx-q/2$ and $q\ll\lambda_N$, so
the assumption $|z+\la_N|>|z+\la_1|$ is not restrictive.
We show the following.

\begin{lemma}\label{3.muopt}
Let $z$ be fixed with $|z+\la_N|>|z+\la_1|$.
Then  $|\wt \eta_z|$,
defined in \eqref{3.wt}, is as small as possible for $\mu_z>-\la_1$
when
\[
\mu_z=-\la_1+\frac{q_z}{1-q_z}\,(\lambda_N-\lambda_1)>-\la_1,
\where q_z:=\biggl|\frac{z+\la_1}{z+\la_N}\biggr|<1.
\]
\end{lemma}
\begin{proof}
It follows from~\eqref{3.wt} that
\[
\wt \eta_z
=-\frac{\sqrt{(z+\la_1)/(z+\la_N)}-\sqrt{(\mu_z+\la_1)/(\mu_z+\la_N)}}
       {\sqrt{(z+\la_1)/(z+\la_N)}+\sqrt{(\mu_z+\la_1)/(\mu_z+\la_N)}},
\]
so with $\xi_1+i\xi_2:=\sqrt{(z+\la_1)/(z+\la_N)}$ and
$\tau:=\sqrt{(\mu_z+\la_1)/(\mu_z+\la_N)}$, we obtain
\[
|\wt \eta_z|^2=\frac{(\xi_1-\tau)^2+\xi_2^2}
{(\xi_1+\tau)^2+\xi_2^2}=1-4\psi(\tau),\where
\psi(\tau):=\frac{\xi_1\tau}{(\xi_1+\tau)^2+\xi_2^2}.
\]
Here, $\xi_1>0$ and we want to choose $\tau>0$ so that $\psi(\tau)$ is 
as large as possible. A short calculation shows that $\psi'(\tau)=0$ implies
$(\xi_1+\tau)^2+\xi_2^2=2(\xi_1+\tau)\tau$, or
$\tau^2=\xi_1^2+\xi_2^2$.  Thus, the maximum
is attained when $(\mu_z+\la_1)/(\mu_z+\la_N)=q_z$, or equivalently when
$\mu_z=-\la_1+(\la_N-\la_1)\,q_z/(1-q_z)$.
\end{proof}

Note that $\mu_z$ tends to $|z+\la_1|-\la_1$ as $\lambda_N$ tends to infinity.

Table \ref{tab3} includes some values of $|\wt \eta_z|$, first 
for the optimal~$\mu_z$ determined by Lemma~\ref{3.muopt}, 
and then (in the final column) for~$\mu_z=0$.  
Comparing the 
$|\wt\eta_z|$ with the corresponding values of~$\wt\ep_z$ in 
Table~\ref{tab1}, we see that, once again, the CG method
is always superior to the Richardson iteration, although in both cases
the preconditioning becomes less effective with increasing~$j$.

\medskip

We now consider a more general preconditioned form of \eqref{1.0}, as in
\eqref{2.10}, where $B_z$ is an Hermitian positive definite operator, so
that the equation may now be written 
\begin{equation}\label{3.gprec}
G_zw=\wt g_z:=B_zg, \where G_z=B_zA_z=zB_z+B_zA.
\end{equation}
Note that $B_z$ and $B_zA$ are Hermitian with respect 
to~$[v,w]:=(B_z^{-1}v,w)$.  We now  define
the Krylov subspaces by
\begin{equation}
\label{3.kryt}
\wt V_n:=\vecspan\{\wt r_0, G_z\wt r_0, \dots, G_z^{n-1}\wt r_0\},
\where \wt r_0:=\wt g_z-G_zw_0=B_zr_0,
\end{equation}
and the CG iterates~$w_n$ by
\begin{equation}\label{6.eq}
(A_z w_n,\fy)
=(g,\fy),\quad\forall\ \fy\in \wt V_n,
\quad\text{where $w_n=w_0+v_n$ with $v_n\in\wt V_n$,}
\end{equation}
or equivalently,
\[
[G_zw_n,\fy]=[\wt g_z,\fy],\quad \forall \fy\in \wt V_n,
\quad\text{where $w_n=w_0+v_n$ with $v_n\in\wt V_n$.}
\]

The existence and uniqueness of~$w_n$ follow as before,
and the inequalities in Lemma~\ref{l4.1} remain valid,
with $\tb{\cdot}$ defined in \eqref{3.tbdef}.
The error again satisfies an orthogonality property,
\[
(A_ze_n,\fy)=0,\quad\forall\fy\in \wt V_n,
\]
and the following quasi-optimality result and its proof carry over verbatim.

\begin{proposition}\label{p6.1}
Let $w$ and $w_n$ be the solutions of \eqref{cg1}~and \eqref{6.eq}.
Then
\[
\tb{w_n-w}\le\sec(\tfrac12\phi)
\inf_{v\in w_0+\wt V_n}\tb{v-w},
\for  \phi=\arg z\in(-\pi,\pi).
\]
\end{proposition}

The proof of the error bound of Theorem~\ref{t4.1} does not remain valid, 
in general, because of the presence of the operator $B_z$ in the 
definition of the Krylov spaces $\wt V_n$, 

%%%%%%%%%%%%%%%%%%%%%%%%%%%%%%%%%%%%%%%%%%%%%%%%%%%%%%%%%%%%%%%%
\section{Practical implementation of the conjugate gradient method}
\label{sec: implementation}
We first derive an algorithm for computing the  iterates~$w_n$ 
in the basic CG method \eqref{cg2} of Section~\ref{sec3}.
In doing so, we make repeated use of the following result.

\begin{lemma}\label{lem: orthog resid}
If $1\le n\le N=\dim(V)$ then the residual~$r_n=g-A_zw_n$ for
\eqref{cg2} satisfies
\[
r_n\in V_{n+1},\andy\
(r_n,\fy)=0,\quad\forall\ \fy\in V_n.
\]
If $r_0\ne0$,
there exists $N^*\le N$ such that $r_n\ne0$ for $0\le n<N^*$, and
$r_n=0$ for $n\ge N^*$.
\end{lemma}
\begin{proof}
The first conclusion is trivial if $r_n=0$, so we may assume $r_n\ne0$.
Since $r_n=g-A_z(w_0+v_n)=r_0-A_zv_n$ and $A_zV_n\subset V_{n+1}$,
we have $r_n\in V_{n+1}$.  The orthogonality property follows at once
from~\eqref{cg2}.
If $r_n=0$ then $w_n=u$ so that, by \eqref{3.ex}~and \eqref{cg2}, $w_j=u$
also for~$j>n$, and thus $r_j=0$ for~$j>n$.
\end{proof}

Lemma~\ref{lem: orthog resid} shows, in particular, that the residuals
$r_0$, $r_1$, \dots, $r_{n-1}$ form an orthogonal basis for the
Krylov space~$V_n$ if $n<N^*$.

We introduce a second sequence of vectors~$p_n$, for $0\le n< N^*$, 
recursively: put $p_0:=r_0$ and, if $p_k\ne0$ 
for $0\le k\le n$, put
\begin{equation}\label{4.pn}
p_{n+1}:=r_{n+1}+\sum_{k=0}^{n}\beta_{nk}p_k,\where\
\beta_{nk}:=-\frac{(A_zr_{n+1},p_k)}{(A_zp_k,p_k)}.
\end{equation}
Here, $\beta_{nk}$ is well-defined since $p_k\ne0$ ensures
$(A_zp_k,p_k)\ne0$.  Also, since $p_0\in V_1$, we have $p_n\in V_{n+1}$
(when defined).
For real $z>0$, the  construction in \eqref{4.pn} amounts to applying
the usual Gramm--Schmidt procedure to construct a new basis for~$V_n$
that is orthogonal with respect to the inner product~$(A_zv,w)$.
For a general complex~$z$, the sesquilinear form~$(A_zv,w)$ is
not an inner product. Even so, we may now show that,
just as for the classical CG method, the sum over~$k$ in \eqref{4.pn}
collapses to include at most one non-zero term.

\begin{lemma}\label{lem: beta_n}
Assume $r_0\ne0$. Then $p_n\in V_{n+1}$ is well defined by \eqref{4.pn} for
$0\le n<N^*$, and $p_n\not\in V_n$, so that
$V_{n+1}=\text{span}\,\{p_0,\dots,p_n\}.$
If $n\ge1$ we have $\beta_{nk}=0$ for $0\le k\le n-1$.
It follows that, recursively,
for $n+1<N^*,$ 
\begin{equation}\label{4.pstep}
p_{n+1}=r_{n+1}+\beta_n p_{n},\where\
\beta_n:=\beta_{n,n}=-\frac{(r_{n+1},A_zp_{n})}{(A_zp_{n},p_{n})}.
\end{equation}
We also have $(A_zp_n,p_k)=0$ for $0\le k\le n-1$, and hence
$(A_zp_n,\fy)=0$ for $\fy\in V$.
\end{lemma}
\begin{proof}
We prove the first statement by induction over~$n$.  
To begin with, note that $p_0\ne0$ and $p_0\in V_1$ since $p_0=r_0$.
Let $1\le n<N^*$ and assume $p_k\ne0$ and $p_k\in V_{k+1}$ 
for~$0\le k\le n-1$,  so that $p_n$  is well-defined by~\eqref{4.pn}.
We cannot have $p_n\in V_n$ because then
$r_n=p_n-\sum_{k=0}^{n-1}\beta_{n-1,k}p_k\in V_n$ and so $r_n=0$
by Lemma \ref{lem: orthog resid}, which would
mean that $n\ge N^*$.

We now observe that, by Lemma~\ref{lem: orthog resid},
\[
(A_zr_{n+1},\varphi)=(z-\bar z)(r_{n+1},\varphi)+(r_{n+1},A_z\varphi)
	=(r_{n+1},A_z\varphi)\quad\text{for all $\varphi\in V_{n+1}$,}
\]
so $\beta_{nk}=0$ for $0\le k\le n-1$, and thus \eqref{4.pstep} holds.
We finally show the last statement by induction on~$n$.  For~$n=1$,
the definition of~$\beta_0$ means that
\[
(A_zp_1,p_0)=\bigl(A_z(r_1+\beta_0p_0),p_0\bigr)
	=(A_zr_1,p_0)+\beta_0(A_zp_0,p_0)=0.
\]
Now let $2\le n<N^*$ and 
assume  that
$(A_zp_{n-1},p_k)=0$ for $0\le k\le n-2$. Then, since
$(A_zr_n,\fy)=0$ for~$\fy\in V_{n-1}$,
\[
(A_zp_n,p_k)=(A_zr_n,p_k)+\beta_{n-1}(A_zp_{n-1},p_k)=0,
	\for 0\le k\le n-2,
\]
and we also have 
$(A_zp_n,p_{n-1})=(A_zr_n,p_{n-1})+\beta_{n-1}(A_zp_{n-1},p_{n-1})=0$.
\end{proof}

Using $w_n$ and $p_n$ we may compute $w_{n+1}$ 
as follows,
and hence  $p_{n+1}$  from \eqref{4.pstep}.

\begin{proposition}\label{l4.3}
If $0\le n<N^*$, then
\[
w_{n+1}=w_n+\alpha_np_n,
\where
\alpha_n:=\frac{\|r_n\|^2}{(A_zp_n,p_n)}.
\]
\end{proposition}
\begin{proof}
Since $w_{n+1}-w_n\in V_{n+1}$ we have $w_{n+1}-w_n=\fy+\alpha_np_n$
for some $\fy\in V_n$ and some scalar~$\alpha_n$.  
Since $(A_zp_n,\fy)=0$ by Lemma \ref{lem: beta_n}, and using \eqref{cg2},
we have
\[
(A_z\fy,\fy)=\bigl(A_z(w_{n+1}-w_n),\fy\bigr)-\al_n(A_zp_n,\fy)=(g,\fy)-(g,\fy)=0,
\]
implying that $\fy=0$.  
For $\al_n$ we find, because $(r_{n+1},r_n)=0$,
\[
\al_n(A_zp_n,r_n)=(A_z(w_{n+1}-w_n),r_n)=
(r_n-r_{n+1},r_n)=\|r_n\|^2.
\]
Here,  since $(A_zp_n,p_{n-1})=0$,
\[
(A_zp_n,r_n)=(A_zp_n,p_n-\beta_{n-1}p_{n-1})=(A_zp_n,p_n),
\]
which shows the value of $\al_n$ stated.
\end{proof}

Note that, by Proposition \ref{l4.3},
\begin{equation}\label{eq: rn recursion}
r_{n+1}=r_n-A_z(w_{n+1}-w_n)=r_n-\al_nA_zp_n,
\end{equation}
so that also the $r_n$ may be
computed recursively. Since $A_zp_n$ needs to be computed anyway
to determine $\al_n$ and $\be_n$ this saves one application
of $A_z$.  We remark that for real $z>0$ the scalar~$\alpha_n$ is real so
$-\alpha_n(r_{n+1},A_zp_n)=(r_{n+1},r_n-\alpha_nA_zp_n)
=\|r_{n+1}\|^2$ and $\beta_n=\|r_{n+1}\|^2/\|r_n\|^2$, which is the
formula used in the classical CG method.

We readily show, using \eqref{eq: rn recursion}~and Proposition~\ref{l4.3},
that 
\[
p_{n+1}=(1+\beta_n)p_n-\alpha_n A_zp_n-\beta_{n-1} p_{n-1},
\]
which is consistent with a result of Faber and 
Manteuffel~\cite[Section~F]{FaberManteuffel1987}:
if a matrix has a complete set of eigenvectors with all eigenvalues lying 
on a line segment in the complex plane, then there exists an inner product
for which the CG iteration yields vectors~$p_n$ that 
satisfy such a three-term recurrence relation.

The algorithm to compute $w_n$ suggested by
Lemma \ref{lem: beta_n} and Proposition \ref{l4.3}
then goes as follows:
Given a preliminary guess $w_0$, compute $r_0=g-A_zw_0$, and set $p_0=r_0$.
The iterative step for $w_n$ and $p_n$ known is then to find first
$w_{n+1}$ from Proposition~\ref{l4.3} and then, 
using \eqref{eq: rn recursion} to
determine $r_{n+1}$, to find $p_{n+1}$ from \eqref{4.pstep}.
The iterations continue until, e.g.,
$\|w_{n+1}-w_n\|$ or $\|r_{n+1}\|$ is bounded by a tolerance, or, 
cf.~Theorem~\ref{t4.1}, this holds for $|\eta_z|^n$.

Consider using this algorithm when~\eqref{cg1} is the linear 
system~\eqref{1.matsg} arising from the semidiscrete, standard Galerkin method 
applied to the heat equation~\eqref{1.heat}.
As before, we have $V=\C^N$, $A=\M^{-1}\S$ and 
$(\vecv,\vecw)=\iprod{\M\vecv,\vecw}$.  Thus, each
application of $A_z$ involves multiplication by~$\M^{-1}$,
however this cost is not incurred in the computation
of $\al_n$ and $\be_n$, since
$(A_z\vecv,\vecw) =z\iprod{\M\vecv,\vecw}+\iprod{\S\vecv,\vecw}.$
\medskip

We now turn to the preconditioned CG method, and consider first the
special preconditioner~$B_z=(\mu_zI+A)^{-1}$, and the
method based on reformulating~\eqref{cg1}
as~\eqref{3.8}, with iterates defined by
\eqref{3.vnt} and \eqref{3.teq}.
The above analysis and the corresponding algorithm  may  be 
applied also in this case. In the iteration step, we now have
$r_n=\wt z B_zg-(\wt z\,I+B_z)w_n$ and in the computation of
$\al_n$ and $\be_n$, the inner product $(A_zv,w)$ is
replaced by $\bigl((\wt z\,I+B_z)v,w\bigr)$. In matrix form,
\eqref{3.8} may be written
\[
\wt z\vecw+(\mu_z\M+\S)^{-1}\M\vecw=\wt z(\mu_z\M+S)^{-1}\vecg,
\]
and for the inner product we have
\[
\wt z(\vecv,\vecw)+(B_z\vecv,\vecw)
=\wt z\iprod{\M\vecv,\vecw}+\iprod{(\mu_z\M+\S)^{-1}\M\vecv,\M\vecw}.
\]
In particular, this method admits a three term recurrence relation,
although the algorithm then requires the application
of~$(\mu_z\M+\S)^{-1}$, which is normally more expensive than that 
of~$\M^{-1}$. This drawback holds also in the case of the lumped mass 
variant of the spatial discretization, where $\M$ is replaced by a
diagonal matrix~$\D$.

\medskip

Although, as noted at the end of Section~\ref{sec3}, the
error analyses of Theorems \ref{t4.1}~and \ref{t3.2} do not carry
over to preconditioned equations of the form \eqref{3.gprec},
we shall
nevertheless proceed to consider the CG method for such equations,
given by \eqref{3.kryt}~and \eqref{6.eq}.  We
derive a recursive  algorithm for computing the~$w_n$, and
in the same way as above first show the following, in which we again
put $r_n=g-A_zw_n$.

\begin{lemma}\label{5.ortres}
The preconditioned residual $\wt r_n:=\wt g_z-G_zw_n=B_zr_n$ satisfies 
\[
\wt r_n\in\wt V_{n+1}\quad\text{and}\quad
[\wt r_n,\fy]=(r_n,\fy)=0\quad
\text{for all $\fy\in \wt V_n$,}\quad\text{for $1\le n\le N$.}
\]
If $r_0\ne0$,
there exists $N^*\le N$ such that 
$r_n\ne0$ for $0\le n<N^*$,
$r_n=0$ for $n\ge N^*$.
\end{lemma}

We define the the sequence~$p_n$ recursively, cf.~\eqref{4.pn}: if 
$p_k\ne0$ for $0\le k\le n$, set
\begin{equation}\label{eq: preconditioned pn}
p_{n+1}:=\wt r_{n+1}+\sum_{k=0}^n\beta_{nk}p_k,\for n\ge0,\with
p_0:=\wt r_0,
\end{equation}
where $\bigl(\be_{n0},\beta_{n1},\ldots,\beta_{nn}\bigr)$ is now
the solution of the lower-triangular, $(n+1)\times(n+1)$ linear system
\begin{equation}\label{6.bk}
\sum_{k=0}^j(A_zp_k,p_j)\beta_{nk}=-(A_z\wt r_{n+1},p_j),\quad
\text{for $0\le j\le n$.}
\end{equation}
The existence and uniqueness of the $\beta_{nk}$ follows since the
diagonal entries $(A_zp_n,p_n)$ are non-zero for~$n<N^*$,
because otherwise $p_n=0$ and we would have 
$\wt r_n\in \wt V_n$ and thus $r_n=0$.
Unfortunately, in contrast to the situation earlier,
$\beta_{nk}\ne0$ is possible for~$k<n-1$,
which requires all the $p_j$ to be stored.
Using the definition \eqref{6.bk} of the $\beta_{nk}$, 
we may now show the following partial
analogue of Lemma~\ref{lem: beta_n}.
 
\begin{lemma}\label{5.pj}
If $r_0\ne0$, then $0\ne p_n\in \wt V_{n+1}$ and
$(A_zp_n,p_k)=0$ for $0\le k<n<N^*$.
\end{lemma}
\begin{proof}
The argument used for Lemma~\ref{lem: beta_n} again establishes
that $0\ne p_n\in \wt V_{n+1}$ for $0\le n<N^*$, and to prove the second
claim we again use finite induction on~$n$:
Taking~$n=0$ in~\eqref{6.bk} gives
$\beta_{00}=-(A_z\wt r_1,p_0)/(A_zp_0,p_0)$
so
\[
(A_zp_1,p_0)=\bigl(A_z(\wt r_1+\beta_{00}p_0),p_0\bigr)
	=(A_z\wt r_1,p_0)+\beta_{00}(A_zp_0,p_0)=0.
\]
Now let $1\le n<N^*$ and assume that $(A_zp_k,p_j)=0$ 
for $0\le j<k\le n$.  For $0\le j\le n$,
\[
(A_zp_{n+1},p_j)=(A_z\wt r_{n+1},p_j)+\sum_{k=0}^j\beta_{nk}(A_zp_k,p_j)
	+\sum_{k=j+1}^{n}\beta_{nk}(A_zp_k,p_j)=0.
\]
This completes the induction step and thus the proof of the lemma.
\end{proof}

As a consequence of Lemma \ref{5.pj}, the conclusion of
Proposition~\ref{l4.3} remains valid: 

\begin{proposition}\label{p4.3}
If $0\le n<N^*$, then
\[
w_{n+1}=w_n+\alpha_np_n,
\where
\alpha_n:=\frac{\db{\wt r_n}^2}{(A_zp_n,p_n)}
	=\frac{(r_n,\wt r_n)}{(A_zp_n,p_n)}.
\]
\end{proposition}
\begin{proof}
The beginning of the proof of Proposition~\ref{l4.3} goes through
verbatim, but since $[\wt r_{n+1},\wt r_n]=0$,
\begin{align*}
\alpha_n(A_zp_n,\wt r_n)&=[G_z(\alpha_np_n),\wt r_n]
	=[G_z(w_{n+1}-w_n),\wt r_n]\\
	&=[\wt r_n-\wt r_{n+1},\wt r_n]=[\wt r_n,\wt r_n]
\end{align*}
and, since $\wt r_n=p_n-\sum_{k=0}^{n-1}\beta_{n-1,k}p_k$,
\begin{align*}
(A_zp_n,\wt r_n)
	=(A_zp_n,p_n)-\sum_{k=0}^{n-1}\bar\beta_{n-1,k}(A_zp_n,p_k)
	=(A_zp_n,p_n).
\end{align*}
\end{proof}

Again, the residuals satisfy
$r_{n+1}=r_n-\alpha_n A_zp_n$, implying that
the preconditioned residuals satisfy
$\wt r_{n+1}=\wt r_n-\alpha_n G_zp_n$.
Each iteration is now more expensive than in the algorithm
proposed by Lemma \ref{lem: beta_n} and Proposition \ref{l4.3},
both in CPU time and memory requirements, and one
may want to restart the iteration every $m$~steps for 
some moderate choice of~$m$.  Figure~\ref{fig: gen precond mod CG}
provides a pseudocode outline of the method in its matrix formulation,
where, as in the discussion following Corollary~\ref{cor: ||Hz-I||},
we let $\A_z=z\M+\S$ and allow $\B_z$ to be any symmetric positive 
definite matrix.  Notice that by working with $\M\vecr_n$ instead 
of~$\vecr_n$, we can avoid computing the action of~$\M^{-1}$.

\begin{figure}
\hrule\vspace{1\jot}
\begin{algorithmic}
\STATE{$\M\vecr_0=\vecg-\A_z\vecw_0$}
\STATE{$\vecp_0=\wt\vecr_0=\B_z\M\vecr_0$}
\FOR{$n=0$ to \textit{max\_iterations}}
\STATE{$\alpha_n=\iprod{\M\vecr_n,\wt\vecr_n}
/\iprod{\A_z\vecp_n,\vecp_n}$}
\STATE{$\vecw_{n+1}=\vecw_n+\alpha_n\vecp_n$}
\STATE{$\M\vecr_{n+1}=\M\vecr_n-\alpha_n\A_z\vecp_n$
(or $\M\vecr_{n+1}=\vecg-\A_z\vecw_n$)}
\STATE{$\wt\vecr_{n+1}=\B_z\M\vecr_n$}
\IF{\textit{converged}}
\STATE{\textbf{break}}
\ENDIF
\STATE{Solve $\sum_{k=0}^j\iprod{\A_z\vecp_k,\vecp_j}\beta_{nk}
=-\iprod{\A_z\wt\vecr_{n+1},\vecp_j}$ for $0\le j\le n$}
\STATE{$\vecp_{n+1}=\wt\vecr_{n+1}+\sum_{k=0}^n\beta_{nk}\vecp_k$}
\ENDFOR
\end{algorithmic}
\vspace{1\jot}\hrule
\caption{Matrix version of CG method for $\A_z\vecw=\vecg$,
preconditioned by~$\B_z$.}
\label{fig: gen precond mod CG}
\end{figure}

%%%%%%%%%%%%%%%%%%%%%%%%%%%%%%%%%%%%%%%%%%%%%%%%%%%%%%%%%%%%%%%%%
\section{A model problem}
\label{sec5}
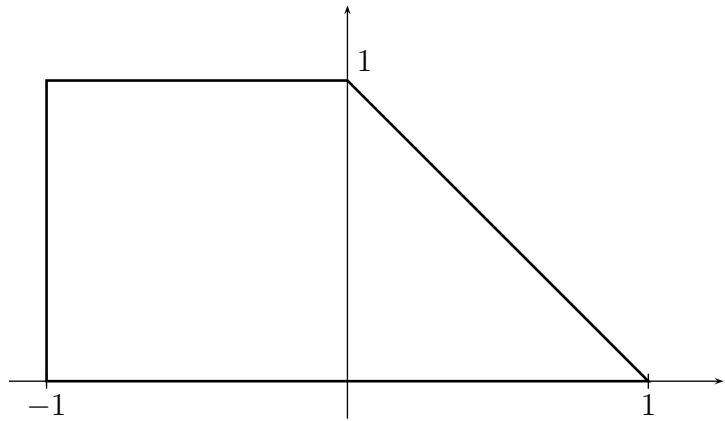
\begin{figure}
\begin{center}
\begin{pspicture}(-5,-1)(5,5)
%\psgrid
\psset{linewidth=1pt}
\pspolygon(-4,0)(4,0)(0,4)(-4,4)
\psset{linewidth=0.5pt}
\psline{->}(-4.5,0)(5,0)
\psline{->}(0,-0.5)(0,5)
\uput[d](4,0){$1$}
\psline(4,-0.1)(4,0.1)
\uput[d](-4,0){$-1$}
\psline(-4,-0.1)(-4,0.1)
\uput[ur](0,4){$1$}
\end{pspicture}
\end{center}
\caption{The domain~$\Omega$.}\label{fig: Omega}
\end{figure}

We now describe a concrete initial boundary-value problem~\eqref{1.heat},
mentioned already in the numerical examples of Sections 
\ref{sec: Richardson}~and \ref{sec3}, and
present some further illustrations of our results.

For the domain~$\Omega$ we took the trapezium with vertices
$(1,0)$, $(0,1)$, $(-1,1)$ and $(-1,0)$, shown in
Figure~\ref{fig: Omega}.  The minimum eigenvalue of~$-\nabla^2$
on~$\Omega$ is close to~$15$, so we chose the diffusivity~$a=1/15$
to give a time scale of order~1 for~\eqref{1.heat}.  
We chose the data $u_0$~and $f$ so that the exact solution is
\[
u(x,y,t)=(1+x)(1-x-y)\sin(\pi y)(1+2t)e^{-t},
\]
and used continuous, piecewise linear finite elements on 
a quasi-uniform, unstructured triangulation~$\T_h$ of~$\Omega$, generated by 
the program Gmsh~\cite{GeuzaineRemacle}.  The dimension of the finite
element space~$V_h$ was $N=2663$, and the maximum element diameter was
$h=0.035$.  The extremal eigenvalues of the operator~$A=\M^{-1}\S$ were
$\lambda_1=1.01380$ and $\lambda_N=4006.79$.
%$\lambda_1=1.013803187544178$ and $\lambda_N=4006.794684302970e+03$.

Table~\ref{tab: discretization error} shows the (discrete) $L_2$-norm
of the error in~$U_{q,h}(t)$ at four values of~$t$, for three choices 
of~$q$, as well as the norm of the solution itself.  We see that once
$q$ is about 20, the $O(h^2)$~error from the spatial discretization
dominates the $O(e^{-q/\log q})$ error from the time discretization;
cf.~\eqref{1.err}.
(Interestingly,
the lumped mass approximation, in which we replace the mass matrix~$\M$
by a diagonal matrix~$\D$, gave slightly more accurate results, with
the added bonus of more favourable extremal eigenvalues: 
$\lambda_1=1.01248$ and $\lambda_N=1387.22$.)  
%$\lambda_1=1.012478711208717$ and $\lambda_N=1387.219505959443e+03$.)  

\begin{table}
\renewcommand{\arraystretch}{1.2}
\begin{center}
\caption{Discretization error $\|U_{q,h}(t)-u(t)\|_h$.}
\label{tab: discretization error}
\begin{tabular}{c|ccc|c}
$t$&$q=10$&$q=20$&$q=30$&$\|u(t)\|_h$\\
\hline
 0.25  &     1.3436e-02   &     4.3778e-04   &     4.1747e-04  &  0.4452 \\
 0.50  &     6.1232e-04   &     1.6260e-04   &     1.7541e-04  &  0.4623 \\
 1.00  &     2.2024e-04   &     2.1088e-04   &     2.1114e-04  &  0.4206 \\
 2.00  &     1.9403e-04   &     1.9411e-04   &     1.9411e-04  &  0.2579 \\
%Using lumped mass approximation
% 0.25  &     1.3311e-02   &     5.5715e-04   &     5.3663e-04  &  0.4452 \\
% 0.50  &     8.2970e-04   &     1.7811e-04   &     1.6456e-04  &  0.4623 \\
% 1.00  &     1.9489e-04   &     2.0405e-04   &     2.0379e-04  &  0.4206 \\
% 2.00  &     1.8966e-04   &     1.8958e-04   &     1.8958e-04  &  0.2579 \\
\end{tabular}
\end{center}
\end{table}

Figure~\ref{fig: modcg conv} shows the convergence history of
the CG method (without preconditioning) when $z=z_j$,
for $j=15$~and $q=20$.  Here, $e_n$ is the \emph{solver} error,
that is, the difference between the $n$th CG iterate and the
exact solution of the discrete problem (as computed using a
direct solver~\cite{Davis2004}).  As well as the $L_2$ error~$\|e_n\|$
and the error~$\tb{e_n}$ in the norm~\eqref{3.tbdef}, we show the
theoretical bound of Theorem~\ref{t3.2}, which is pessimistic but
with roughly the correct error reduction factor.

\begin{figure}
\begin{center}
\includegraphics[scale=0.55]{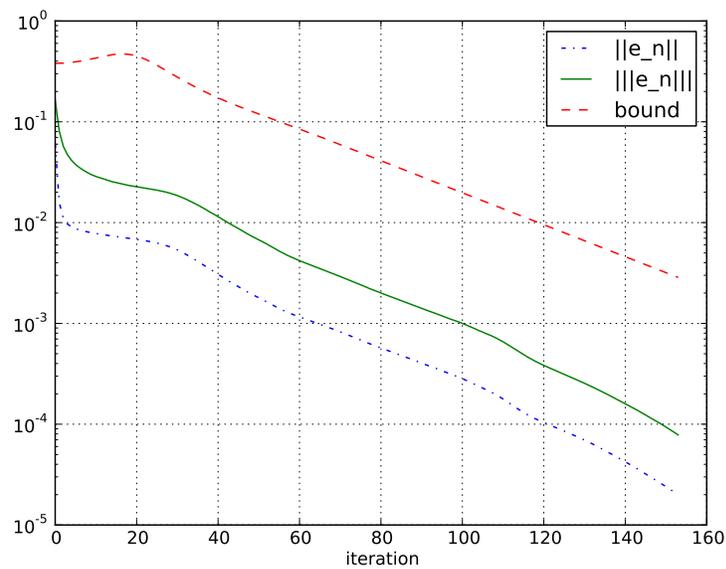}
\end{center}
\caption{Convergence history of CG method (no preconditioning).}
\label{fig: modcg conv}
\end{figure}

\begin{table}
\renewcommand{\arraystretch}{1.2}
\newcommand{\num}{\phantom{0}}
\begin{center}
\caption{Iteration counts at different quadrature points.}
\label{tab: iterations}
\begin{tabular}{r|cc|rcrc|cc}
&\multicolumn{2}{c|}{Richardson}&
\multicolumn{4}{c|}{CG}&
\multicolumn{2}{c}{}\\
\hline
$j$& INV& AMG(3)&\multicolumn{1}{c}{--}& INV& IC& AMG(1)&
$\|w_j\|$& $\epsilon_j$\\
\hline
  0&  \num1&  \num5&    250&  \num1&     52&  \num7& 1.14e+00&   3.18e-06\\
%  1&      5&      8&    224&      5&     47&      7&   1.14e+00&   3.14e-06\\
  2&  \num7&  \num9&    227&  \num5&     48&  \num7& 1.13e+00&   3.06e-06\\
%  3&      8&     12&    231&      6&     49&      7&   1.09e+00&   2.94e-06\\
  4&     10&     15&    235&  \num6&     50&  \num8& 1.03e+00&   2.84e-06\\
%  5&     13&     19&    240&      7&     50&      8&   9.18e-01&   2.78e-06\\
  6&     15&     25&    242&  \num7&     51&  \num9& 7.67e-01&   2.78e-06\\
%  7&     18&     33&    241&      8&     50&      9&   5.97e-01&   2.86e-06\\
  8&     24&     42&    234&  \num8&     50&     10& 4.39e-01&   3.03e-06\\
%  9&     31&     50&    226&      9&     48&     10&   3.13e-01&   3.35e-06\\
 10&     39&     56&    219&  \num9&     46&     11& 2.21e-01&   3.86e-06\\
% 11&     45&     58&    210&     10&     44&     11&   1.60e-01&   4.70e-06\\
 12&     49&     57&    184&     10&     40&     11& 1.19e-01&   6.08e-06\\
% 13&     49&     52&    167&      9&     36&     11&   9.23e-02&   8.43e-06\\
 14&     48&     45&    149&  \num9&     32&     10& 7.41e-02&   1.27e-05\\
% 15&     47&     40&    124&      9&     28&     10&   6.10e-02&   2.09e-05\\
 16&     44&     37&     98&  \num8&     22&  \num9& 5.11e-02&   3.83e-05\\
% 17&     39&     32&     65&      6&     17&      7&   4.33e-02&   7.95e-05\\
 18&     32&     26&     34&  \num5&     11&  \num5& 3.69e-02&   1.91e-04\\
% 19&     23&     18&      6&      2&      5&      2&   3.16e-02&   5.42e-04\\
 20&  \num8&  \num7&     10&  \num2&      3&  \num2& 2.71e-02&   1.87e-03\\
\end{tabular}
\end{center}
\end{table}

In Table~\ref{tab: iterations} we show iteration counts at 
alternate quadrature points for several versions of the Richardson
and CG iterations.  In the column headings, INV refers to the
the special preconditioner $\B_z=(\mu_z\M+\S)^{-1}$, AMG($k$) refers to
the algebraic multigrid preconditioner~\cite{BellOlsonSchroder2011} 
with $k$~V-cycles, and IC refers to the
incomplete Cholesky preconditioner~\cite{JonesPlassmann1995}.  
The first CG column shows the results
using no preconditioner.  As the acceleration
parameter, we chose $\alpha=\alpha_z$ from Theorem~\ref{th2.3} 
in the case of the INV preconditioner, and $\alpha=\alphasz$ from
Theorem~\ref{2t.prec} for AMG(3).  
For both sets of Richardson iterations, we 
chose $\mu_z$ as in Tables \ref{tab1}--\ref{tab: Richardson AMG}, to 
minimize $\wt\ep_z$ from Theorem~\ref{th2.3}.  
Likewise, for all of the preconditioned CG 
iterations we chose the optimal value of~$\mu_z$ for the INV preconditioner, 
given in Lemma~\ref{3.muopt}.  Except for~$j=0$, the AMG(1) preconditioner
for CG is almost as effective as INV, requiring only 11~iterations in the
worst case.  One could also reduce the setup cost for AMG by using the
same~$\mu_z$ for several nearby quadrature points, but we did not
investigate the tradeoff between the cost saving and a possibly slower
convergence.

As the stopping criterion, we used
\begin{equation}\label{eq: epsilon_j delta}
\|e_n\|\le\epsilon_j\quad\text{where}\quad 
\epsilon_j:=\delta\times\frac{e^{-\Re(z_j)t}}{(q+1)k|z'_j|}
\quad\text{for $\delta=10^{-5}$ and $t=1$.}
\end{equation}
In this way, the estimate~\eqref{eq: E(t)} ensures that the additional
error in $U_{q,h}(t)$ due to the iterative solver is less than~$\delta$.
For $j=0$, we started each iteration with the zero vector, but
for~$j\ge1$, we used the final iterate at~$z_{j-1}$ as the starting
iterate at~$z_j$.
The remaining columns of the table show the values 
of~$\|w_h(z_j)\|$~and $\epsilon_j$.   Since the former are decreasing
and the latter are increasing, the stopping criterion becomes easier 
to satisfy with increasing $j$, overcoming
the deterioration in the error reduction factors of the iterative solvers,
seen in Tables \ref{tab1}, \ref{tab2}~and \ref{tab3}.

%%%%%%%%%%%%%%%%%%%%%%%%%%%%%%%%%%%%%%%%%%%%%%%%%%%%%%%%%%%%%%%%%
%\newpage
\bibliographystyle{plain}
\bibliography{mclean_thomee_pp-refs}

\begin{thebibliography}{10}

\bibitem{BellOlsonSchroder2011}
W.~N. Bell, L.~N. Olson, and J.~B. Schroder.
\newblock {\em PyAMG: Algebraic Multigrid Solvers in Python v2.0}, 2011.

\bibitem{BenziBertaccini2008}
M.~Benzi and D.~Bertaccini.
\newblock Block preconditioning of real-valued iterative algorithms for complex
  linear systems.
\newblock {\em IMA J. Numer. Anal.}, 28:598--618, 2008.

\bibitem{Br}
J.~H. Bramble.
\newblock {\em Multigrid Methods}, volume 294 of {\em Pitman Research Notes in
  Mathematics}.
\newblock Pitman, 1993.

\bibitem{Davis2004}
T.~A. Davis.
\newblock Algorithm~832: Umfpack, an unsymmetric-pattern multifrontal method.
\newblock {\em ACM Trans. Math. Software}, 30:196--199, 2004.

\bibitem{FaberManteuffel1987}
V.~Faber and T.~A. Manteuffel.
\newblock Orthogonal error methods.
\newblock {\em SIAM J. Numer. Anal.}, 24:170--187, 1987.

\bibitem{Freund1992}
R.~W. Freund.
\newblock Conjugate gradient-type methods for linear systems with complex
  symmetric coefficient matrices.
\newblock {\em SIAM J. Sci. Stat. Comput.}, 13:435--448, 1992.

\bibitem{GM}
I.~P. Gavrilyuk and V.~L. Makarov.
\newblock Exponentially convergent algorithms for the operator exponential with
  applications to inhomogeneous problems in banach spaces.
\newblock {\em SIAM J. Numer. Anal.}, 43:2144--2171, 2005.

\bibitem{GeuzaineRemacle}
C.~Geuzaine and J.-F. Remacle.
\newblock {\em Gmsh: a three-dimensional finite element mesh generator with
  built-in pre- and post-processing facilities}.

\bibitem{scipy}
Eric Jones, Travis Oliphant, Pearu Peterson, et~al.
\newblock {\em SciPy: Open Source Scientific Tools for Python}, 2001--{}.

\bibitem{JonesPlassmann1995}
M.~T. Jones and P.~E. Plassmann.
\newblock Algorithm 740: Fortran subroutines to compute improved incomplete
  cholesky factorizations.
\newblock {\em ACM Trans. Math. Software}, 21:18--19, 1995.

\bibitem{biv06}
W.~McLean, I.~H. Sloan, and V.~Thom\'ee.
\newblock Time discretization via {L}aplace tranformation of an
  integro-differential equation of parabolic type.
\newblock {\em Num. Math.}, 102:497--522, 2006.

\bibitem{bivi04}
W.~McLean and V.~Thom\'ee.
\newblock Time discretization of an evolution equation via {L}aplace
  transforms.
\newblock {\em IMA J. Numer. Anal.}, 24:439--463, 2004.

\bibitem{bivi08}
William McLean and Vidar Thom\'ee.
\newblock Maximum-norm error analysis of a numerical solution via {L}aplace
  transformation and quadrature of a fractional order evolution equation.
\newblock {\em IMA J. Numer. Anal.}, 30:208--230, 2010.

\bibitem{bivi07}
William McLean and Vidar Thom\'ee.
\newblock Numerical solution via {L}aplace transforms of a fractional order
  evolution equation.
\newblock {\em J. Integral Equations Appl.}, 22:57--94, 2010.

\bibitem{SST99}
D.~Sheen, I.~H. Sloan, and V.~Thom\'ee.
\newblock A parallel method for time-discretization of parabolic equations
  based on contour integral representation and quadrature.
\newblock {\em Math. Comp.}, 69:177--195, 1999.

\bibitem{SST03}
D.~Sheen, I.~H. Sloan, and V.~Thom\'ee.
\newblock A parallel method for time-discretization of parabolic equations
  based on {L}aplace transformation and quadrature.
\newblock {\em IMA J. Numer. Anal.}, 23:269--299, 2003.

\bibitem{th05}
V.~Thom\'ee.
\newblock A high order parallel method for time discretization of parabolic
  type equations based on {L}aplace transformation and quadrature.
\newblock {\em Int. J. Numer. Anal. Model.}, 2:85--96, 2005.

\bibitem{th06}
V.~Thom\'ee.
\newblock {\em Galerkin Finite Element Methods for Parabolic Problems}.
\newblock Springer-Verlag, Berlin, second edition, 2006.

\end{thebibliography}
\end{document}